\documentclass[12pt]{article}
\usepackage{amsfonts,amsbsy,amssymb,amsmath,amsthm}
\usepackage[toc,title,page]{appendix}
\usepackage{a4}
\usepackage{graphicx}
\usepackage[english]{babel}
\usepackage{datetime}
\usepackage[colorinlistoftodos]{todonotes}
\usepackage[numbers,sort]{natbib}
\usepackage{varioref}
\usepackage{hyperref}
\usepackage{cleveref} % \Cref{thm}

\usepackage{algorithm}
\usepackage{algorithmic}
\usepackage{caption}
\usepackage{subcaption}

\newcommand{\diag}{\operatorname{diag}}

\usepackage{comment}
\DeclareMathOperator{\Tr}{Tr}
\newcommand{\refe}{\sigma}
\newcommand{\temp}{\mu}
\newcommand{\<}{\langle}
\renewcommand{\>}{\rangle}

\newcommand{\RR}{\mathbb{R}}

\newcommand{\HH}{\mathcal{H}}

\newcommand{\wspaceR}{\mathcal{P}_2(\mathbb{R}^n)}
\usepackage{bbm}
\usepackage{multirow}

\newtheorem{theorem}{Theorem}
\newtheorem{definition}[theorem]{Definition}
\newtheorem{example}[theorem]{Example}
\newtheorem{lemma}[theorem]{Lemma}
\newtheorem{remark}{Remark}

\newtheorem{corollary}[theorem]{Corollary}

\title{Supervised learning of sheared distributions using linearized optimal transport}
\author{Varun Khurana\thanks{Department of Mathematics, University of California, San Diego, CA
  (vkhurana@ucsd.edu, hkannan@ucsd.edu, acloninger@ucsd.edu, cmoosmueller@ucsd.edu)} \and Harish Kannan\footnotemark[1] \and
 Alexander Cloninger\footnotemark[1] \thanks{Halicio{\u g}lu Data Science Institute, University of California, San Diego, CA} \and Caroline Moosm\"uller\footnotemark[1]}
\date{}
\begin{document}

\maketitle

%%%%%%%%%% ABSTRACT %%%%%%%%%%%%%%%%%%%%%%%%%%%%%%%

\begin{abstract}
In this paper we study supervised learning tasks on the space of probability measures. We approach this problem by embedding the space of probability measures into $L^2$ spaces using the optimal transport framework. In the embedding spaces, regular machine learning techniques are used to achieve linear separability. This idea has proved successful in applications and when the classes to be separated are generated by shifts and scalings of a fixed measure. This paper extends the class of elementary transformations suitable for the framework to families of shearings, describing conditions under which two classes of sheared distributions can be linearly separated. We furthermore give necessary bounds on the transformations to achieve a pre-specified separation level, and show how multiple embeddings can be used to allow for larger families of transformations. We demonstrate our results on image classification tasks.

\end{abstract}

%\todo[inline]{Major things to finish: 1) Fix up the shearing corollary.  2) Add the binary classification with multiple references part.  3) Add the proofs in the appendix A for the section 3.  4) add proofs in appendix B for section 4 and 5.}

%%%%%%%%% INTRO %%%%%%%%%%%%%%
\section{Introduction}

%% LOT INTRO %%%

We consider the problem of classifying probability measures $\mu_i$ on $\RR^n$ based on a finite set of pre-classified training data $\{(\mu_i, y_i)\}_{i=1}^N$, where $y_i$ denote the labels. The aim is to use the given training data to build a function $f$ that assigns a probability measure to its correct label, i.e.\ we study supervised learning techniques on the space of probability measures.

The problem of classifying probability measures rather than points in $\RR^n$ has a number of applications, a few examples being classification of population groups \cite{cloninger2019people}, and classification of flow cytometry and other measurements of cell or gene populations per person \cite{bruggner2014automated, cheng2017two, zhao2020detecting}.  
Note that for application purposes, we need to consider samples of probability measures $\mu_i$, hence the task requires to meaningfully compare and classify point clouds.

The largest issue associated with this classification problem is the generation of features of $\mu_i$ that can be used to build a classifier $f$. Many methods use an embedding idea to transform the set of probability measures into a Hilbert space in which regular machine learning techniques can be applied for the classification task, e.g.\ embeddings through moments or kernels \cite{muandet2016kernel,newey1987hypothesis}.

In this paper we are interested in such embeddings based on the optimal transport framework \cite{villani-2009}. Optimal transport gives rise to a natural distance on the space of probability measures via the Wasserstein distance, which quantifies the minimal work necessary to move one distribution into another using an optimal transport plan. Optimal transport has gained high interest in the machine learning community in recent years, for example for generative models, semi-supervised learning or imaging applications \cite{arjovsky2017wasserstein,rubner2000earth,solomon2014wasserstein}.

We use the optimal transport plan or map to build an embedding of probability measures into an $L^2$-space known as ``Linear Optimal Transportation'' (LOT) \cite{wei13,park18, aldroubi20, moosmueller20,gigli-2011} or ``Monge embedding'' \cite{merigot20}.  
LOT is a set of transformations based on optimal transport maps, which map a distribution $\mu$ to the optimal transport map that takes a fixed reference distribution $\sigma$ to $\mu$:
\begin{align}\label{intro:lot}
\mu \mapsto T_\sigma^{\mu}, & &\textnormal{ where } T_\sigma^{\mu}:= \arg\min_{T\in \Pi_\sigma^\mu} \int \|T(x)-x\|_2^2 d\sigma(x),
\end{align}
where $\Pi_\sigma^\mu$ denotes the set of measure preserving maps from $\sigma$ to $\mu$. Through the embedding \eqref{intro:lot}, the optimal transport map to a fixed reference $\sigma$ is used as a feature of $\mu$.
Note that LOT takes the nonlinear space of probability measures into a linear (but infinite dimensional) space of $L^2$ functions. This makes LOT particularly interesting as a feature space. Indeed, it has been demonstrated in various applications that within the LOT embedding space, classes of probability measures can be well separated with linear machine learning tools. The main applications concern signal and image classiﬁcation tasks \cite{kolouri-2016,park18,moosmueller20}, such as distinguishing facial expressions, separating healthy from cancerous tissue classes \cite{wang11}, and visualizing phenotypic differences between types of cells \cite{basu14}.

While the LOT embedding space is well studied in $1$-dimension \cite{park18}, since LOT can be thought of as a generalized CDF, many questions remain open in higher dimensions. This has to do with the fact that in higher dimensions, there is a large family of potential group actions that can be applied to a distribution $\mu_i$ (e.g., shifts, scalings, shearings, rotations), and $\Pi_\sigma^\mu$ contains a large number of measure preserving maps.

It has been shown that shifts and scalings behave well with respect to the LOT embedding \cite{aldroubi20,moosmueller20,park18}, meaning that two classes of probability measures obtained from scaling or shifting of a fixed measure can be linearly separated in the LOT embedding space. The reason lies in a property we refer to as the ``compatibility condition'', which is satisfied by shifts and scalings \cite{aldroubi20,moosmueller20}. This property describes an interplay between LOT and the pushforward operator, or in terms of Riemannian geometry, the invertability of the exponential map \cite{gigli-2011}.  Similarly, small perturbations of the distributions in these classes can still be linearly separated under certain minimal separation conditions \cite{moosmueller20}.

The contributions of this paper are threefold. We first describe conditions under which families of shearings satisfy the compatibility condition, enlarging the space of functions for which linear classification results hold in the LOT embedding space (\Cref{sec:shears}). The second contribution concerns binary classification results with pre-specified level of separation (\Cref{sec:delta_separation}). We give necessary bounds on the classes of probability measures to achieve linear separation in the embedding space with given separation level. The bounds are in terms of the parameters associated with the set of elementary transformations that are used to create the two classes. In the third part (\Cref{sec:multiple_ref}), we study embeddings using multiple references. Based on the set of elementary transformations, we quantify the number references needed to achieve a desired separation level in the embedding space. The paper closes with classification experiments on sheared distributions.

%\subsection{Main contributions}

%\subsection{Notation}

%\todo[inline]{Rewrite because this is from Caroline's paper}

%%%%%% PRELIMS %%%%%%%%%%%%%%%%%%
\section{Tools from optimal transport}
%\todo[inline]{Write later}

% \begin{enumerate}
%     \item introduce what OT is (smooth problem)
%     \item talk about LOT
%     \item compatibility condition and the results of \cite{aldroubi20}
%     \item classification, citing results from \cite{aldroubi20,park18,wei13,moosmueller20}, stress again the importance of the compatibility condition for these results
% \end{enumerate}

This paper deals with probability measures on $\RR^n$, i.e.\ with elements of the space $\mathcal{P}(\RR^n)$. We mostly deal with probability measures that have bounded second moment, and denote the respective space by $\mathcal{P}_2(\RR^n)$. The Lebesgue measure is denoted by $\lambda$.
% denote the probability measures in $\mathcal{P}(\RR^n)$ with bounded second moment.  In particular, if $\refe \in \mathcal{P}(\RR^n)$ satisfies
% \begin{equation*}
%     \int \Vert x \Vert_2^2 d\refe(x) < \infty.
% \end{equation*}

To any probability measure $\refe$, we assign the function space $L^2(\RR^n, \refe)$, which is equipped with the $L^2$-norm with respect to $\refe$: 
\begin{equation*}
    \Vert f \Vert_\refe^2 = \int \Vert f(x) \Vert_2^2 \, d\refe(x).
\end{equation*}

%% PART ON DENSITIES %% 

If a measure $\refe$ is absolutely continuous with respect to $\lambda$, written as $\refe \ll \lambda$, then there exists a density $f_\refe: \RR^n \to \RR$ such that
\begin{equation*}
    \refe(A) = \int_A f_\refe(x) d\lambda(x),
\end{equation*}
with $A\subseteq \RR^n$ measurable. For the most part, the probability measures we consider are absolutely continuous with respect to $\lambda$.

%% DO WE NEED THAT MEASURES ARE ABOLUTELY CONTINUOUS? - YES

A function $S: \RR^n \to \RR^n$ gives rise to the pushforward measure of $\refe$:
\begin{equation}\label{eq:pushforward}
    S_\sharp \refe(A) = \refe(S^{-1}(A)).
\end{equation}
where $A \subset \RR^n$ measurable. Throughout this paper, we denote the Jacobian of a function $S$ by $J_S$.

%% PUSHFORWARD IN TERMS OF DENSITIES NECESSARY?

% If $\refe \ll \lambda$, then in terms of densities, the pushforward relation $\nu(A) = \refe(S^{-1} (A))$ is given by
% \begin{equation*}
%     \int_{S^{-1}(A)} f_\refe (x) d\lambda(x) = \int_A f_\nu (y) d\lambda(y), \hspace{0.2cm} A \subseteq \RR^n \text{ measurable}.
% \end{equation*}
%We will denote the reference distribution by $\refe$ and the template distribution by $\temp$.

Given two measures, $\refe$ and $\temp$ there may exist many maps $S$ such that $S_{\sharp}\refe = \temp$. In order to find a unique map that pushes $\refe$ into $\temp$, the theory of optimal transport \cite{villani-2009} imposes an ``optimality condition'' on the map $S$. It has to minimize the overall cost of pushing $\refe$ into $\temp$, where cost is measured by a metric in the underlying space (here we use the Euclidean distance in $\RR^n$):
\begin{equation}\label{eq:cost}
    \int \Vert S(x) - x \Vert_2^2 d\refe(x).
\end{equation}
If such a cost minimizing function exists, then
\begin{equation}\label{eq:Wasserstein_distance}
    W_2(\refe, \temp)^2 = \min_{S: S_\sharp \refe = \temp} \int \Vert S(x) - x \Vert_2^2 d\refe(x).
\end{equation}
is called the \emph{Wasserstein-2} distance between $\refe$ and $\temp$. Note that the Wasserstein problem can also be considered for different norms (like $p$-norm) and on Riemannian manifolds \cite{brenier-1991,villani-2009,mccann-2001,ambrosio-2013}.

Brenier's theorem \cite{brenier-1991} states that under the assumption of $\refe \ll \lambda$, a unique map exists that pushes $\refe$ into $\temp$ and minimizes \eqref{eq:cost}. We call this map ``the optimal transport from $\refe$ to $\temp$'' and denote it by $T_{\refe}^{\temp}$.

We furthermore make use of the following result:
\begin{theorem}[Brenier's theorem \cite{brenier-1991}]\label{Brenier} 
If $\refe \ll \lambda$, the optimal transport map $T_{\refe}^{\temp}$ is uniquely defined as the gradient of a convex function $\varphi$, i.e.\ $T_{\refe}^{\temp}(x) = \nabla \varphi(x)$, where $\varphi$ is the unique convex function that satisfies $(\nabla \varphi)_\sharp \refe = \temp$. Uniqueness of $\varphi$ is up to an additive constant.
\end{theorem}

\subsection{Linear optimal transport embeddings}
In this section, we introduce linear optimal transport embeddings, as proposed by \cite{wei13,park18,gigli-2011}. A fixed \emph{reference measure} $\refe$ gives rise to an embedding of $\wspaceR$ into $L^2(\RR^n,\refe)$ via the map 
\begin{equation}\label{eq:LOT}
    \temp \mapsto T_{\refe}^{\temp}.
\end{equation}
We denote this map by $F_{\refe}$, and call it ``LOT'' or ``LOT embedding'' (sometimes $F_{\refe}$ is called \emph{Monge map} as well \cite{merigot20}). 

Note that the nonlinear space of measures is mapped into a linear (infinite dimensional space) of $L^2$ functions. This makes LOT particularly interesting as a feature space, for example to use linear machine learning techniques to classify subsets of $\wspaceR$ \cite{moosmueller20,park18}. Other fields of application include the approximation the Wasserstein distance with a linear $L^2$-distance \cite{moosmueller20,merigot20}, and fast barycenter computation and clustering \cite{merigot20}.

From a theoretical point of view, the regularity of \eqref{eq:LOT} has been studied in \cite{merigot20,gigli-2011}. Indeed, the H\"older regularity of \eqref{eq:LOT} is not better than $1/2$. We also mention the results of \cite{berman20}, where a map related to LOT is analyzed, namely $\refe \mapsto T_{\refe}^{\temp}$. 

A central property in the study of LOT is the so-called \emph{compatibility condition} \cite{moosmueller20,aldroubi20}. It describes an interplay between LOT and the pushforward operator \eqref{eq:pushforward}.

\begin{definition}\label{def:compatibility}
Fix $\refe,\temp \in \wspaceR$ with $\refe \ll \lambda$. The LOT embedding $F_{\refe}$ is called \emph{compatible} with the $\temp$-pushforward of a function $S \in L^2(\RR^n,\temp)$ if
\begin{equation*}
    F_{\refe}(S_{\sharp}\temp) = S\circ F_{\refe}(\temp).
\end{equation*}
\end{definition}
Note that the compatibility condition of \Cref{def:compatibility} can also be written as 
\begin{equation*}%\label{eq:compatibility}
    T_{\refe}^{S_{\sharp}\temp} = S\circ T_{\refe}^{\temp}.
\end{equation*}
Considering the Riemannian manifold $(\wspaceR,W_2)$ with exponential map (the pushforward operator), LOT can be viewed as its right-inverse. % not sure if its necessary to cite \cite{gigli-2011}
For $\refe=\temp$, the compatibility condition forces LOT to be a left-inverse as well.

Under the assumption of the compatibility condition, a series of interesting results can be derived. 
First, the Wasserstein-2 distance can be computed from the linear $L^2$-distance,
\begin{equation}\label{eq:linear_W2}
    W_{2,\refe}^{\operatorname{LOT}}(\temp_1,\temp_2) := \Vert F_\refe(\temp_1) - F_\refe(\temp_2) \Vert_\refe 
    %= \Vert T_\refe^{\temp_1} - T_\refe^{\temp_2} \Vert_\refe = \bigg( \int_{\RR^n} \Vert T_\refe^{\temp_1}(x) - T_\refe^{\temp_2}(x) \Vert_2^2 d\refe(x) \bigg)^{1/2}.
\end{equation}
if $\temp_1,\temp_2$ have been obtained from a fixed \emph{template} $\temp$ via pushforwards of two functions $S_1,S_2$ for which the compatibility condition holds \cite{moosmueller20}, i.e.\ in this case
\begin{equation}
W_2({S_1}_{\sharp}\temp,{S_2}_{\sharp}\temp)=W_{2,\refe}^{\operatorname{LOT}}({S_1}_{\sharp}\temp,{S_2}_{\sharp}\temp).   \end{equation}
This is of particular interest when trying to compute the pairwise distance between many measures $\{\temp_i\}_{i=1,\ldots,N}$, when each $\mu_i$ is obtained from a fixed template $\temp$ via the process $\temp_i={S_i}_{\sharp}\temp$ with compatible functions $\{S_i\}_{i=1,\ldots, N}$ (\cite{aldroubi20} calls such a process an ``algebraic generative model''). In this setting, one can compute the $N$ transport maps $T_{\refe}^{\temp_i}$, and then compute ${N\choose 2}$ linear distances via \eqref{eq:linear_W2}, which is computationally much cheaper (especially for large $N$), than computing ${N \choose 2}$ transport maps (Wasserstein-2 distances). These results also generalize to when the compatibility condition is only satisfied up to an error $\varepsilon>0$ \cite{moosmueller20}. Then the linear distance \eqref{eq:linear_W2} approximates $W_2$ up to an error of order $\varepsilon^{1/2}$. Other approximation results (that do not need the compatibility condition) can be found in \cite{merigot20}.

Second, under the assumption of the compatibility condition, convexity is preserved under LOT \cite{aldroubi20,moosmueller20}. In particular, if $\HH \subseteq L^2(\RR^n,\refe)$ is a set of convex and compatible functions, then $F_{\refe}(\HH)$ is also convex (a similar results holds for almost convex sets \cite{moosmueller20}). The preservation of convexity is crucial to deduce linear separability results in the embedding space through the Hahn-Banach theorem (e.g.\ to apply LOT in supervised learning). Indeed it has been shown that under the assumption of the compatibility condition, binary classification of sets of probability measures can be achieved in the LOT embedding space with linear methods, i.e.\ in the embedding space, a separating hyperplane can be found \cite{moosmueller20,park18}.

Yet the compatibility condition (\Cref{def:compatibility}) is very restrictive, and cannot be expected to hold for all $S$. As of now, it is known that shifts and scalings, i.e.\ functions of the form $S(x)=cx+b$ with $c>0$ and $b\in \RR^n$, satisfy \Cref{def:compatibility} for all choices of $\refe,\temp$ \cite{park18,aldroubi20,moosmueller20}. \cite{aldroubi20} also shows that for fixed $\refe$, for the compatibility condition to hold for all $\temp$, $S$ has to be a shift/scaling.

It is our aim to extend the set of compatible functions $S$ beyond shifts and scalings to make LOT applicable to a broader range of applications. In particular we study (generalized) affine transformations. Note that because of the result in \cite{aldroubi20}, to increase the set of compatible functions, the reference $\refe$ and the template $\temp$ can no longer be chosen independently. In the next section we establish necessary relationships between $\refe,\temp$ and $S$ for \Cref{def:compatibility} to hold.

% Given a reference distribution $\refe$, we define the Linear Optimal Transport (LOT) embedding as $F_\refe: \mathcal{P}_2(\RR^n) \to L^2(\RR^n, \refe)$, where $F_\refe(\temp) = T_\refe^\temp$.  The LOT distance with respect to $\refe$ between measures $\temp_1$ and $\temp_2$ is defined as
% \begin{equation*}
%     \Vert F_\refe(\temp_1) - F_\refe(\temp_2) \Vert_\refe = \Vert T_\refe^{\temp_1} - T_\refe^{\temp_2} \Vert_\refe = \bigg( \int_{\RR^n} \Vert T_\refe^{\temp_1}(x) - T_\refe^{\temp_2}(x) \Vert_2^2 d\refe(x) \bigg)^{1/2}.
% \end{equation*}

%%%%%%%% RESULT ON COMPATIBILITY CONDITION %%%%%%%%%%%%%%%

\section{Compatibility condition for affine transformations}\label{sec:shears}

In this section we study the conditions under which affine transformations $S(x) = Ax+b$ (and generalizations of such transformations) satisfy the compatibility condition (\Cref{def:compatibility}). Our results show that fixing the reference $\refe$ and template $\temp$ generates necessary conditions for maps $S$ to satisfy the compatibility conditions with respect to $\refe$ and $\temp$.  Conversely, fixing the template $\temp$ and the transformations $S$ generates necessary conditions that references $\refe$ must satisfy in order for the compatibility condition to hold.  
These results strongly depend on the following theorem.

%%% theorem statement %%%%
\begin{theorem}[Informal Statement of Theorem \ref{main}]\label{mainTheorem}
Let $\refe, \temp \in \mathcal{P}_2(\RR^n)$ and let $\refe \ll \lambda$. Let $S \in C^1(\RR^n,\RR^n)$ such that $S = \nabla \varphi$ for some twice differentiable function $\varphi$. We also assume that $S$ satisfies the compatibility condition (\Cref{def:compatibility}). Then the Jacobian of $S$, $J_S$, is symmetric positive definite and shares the same eigenspaces as the Jacobians of $T_\refe^\temp$ and $T_\refe^{S_\sharp \temp}$.
\end{theorem}
\begin{proof}
The proof can be found in \Cref{proofOfMain}.
\end{proof}

We get the following corollary.

\begin{corollary}
Let $\refe, \mu \in \mathcal{P}_2(\RR^n)$ and let $\refe \ll \lambda$.  If $S \in C^1(\RR^n, \RR^n)$ such that $S = \nabla \varphi$ for some twice differentiable $\varphi$ and $S$ satisfies the compatibility condition for $\refe$ and $\temp$, then $S$ is an optimal transport map.
\end{corollary}
\begin{proof}
In particular, note that \Cref{mainTheorem} states that if $S = \nabla \varphi$ for some $\varphi$; and if the compatibility condition holds, then $\nabla^2 \varphi$ is positive definite.  Thus, $\varphi$ must have been convex.  In light of Brenier's theorem \Cref{Brenier}, $S$ must be an optimal transport map.  Informally, \Cref{mainTheorem} above states that this optimal transport map $S$ must be transporting mass in the same directions (eigenspaces) as $T_\refe^\mu$.
\end{proof}

We use \Cref{mainTheorem} above to extend a form of LOT isometry to the case when $S$ is an affine transformation.  The only caveat for our extension is that the orthonormal basis on which we shear must be constant.  The relevant function class for this setting is given in the following definition.

\begin{definition}\label{def:shears}
Given an orthogonal matrix $P \in \RR^{n \times n}$, define the \emph{constant orthonormal basis shears} as the class of maps
\begin{align*}
    \mathcal{F}(P) = \Bigg\{ x \mapsto \Tilde{P}^\top \begin{bmatrix} f_1( (\Tilde{P}x)_1) \\ f_2 ( (\Tilde{P}x)_2) \\ \vdots \\ f_n ( (\Tilde{P}x)_n)) \end{bmatrix} + b : \substack{f_j: \RR \to \RR \text{ is monotonically} \\ \text{increasing and differentiable} \\
    \text{and $b \in \RR^n$}} \Bigg\},
\end{align*}
where $\Tilde{P}$ is a row-permutation of the orthogonal matrix $P$.
\end{definition}

Note that affine transformations $S(x)=Ax+b$ with $A=P^TD P$ and $d_i>0,i=1,\ldots, n$ (i.e.\ symmetric positive definite matrices diagonalizable by $P$), are elements of $\mathcal{F}(P)$. Indeed, choose $f_i(y) = d_iy, i=1,\ldots,n$.

% It was shown in \cite{moosmueller20} that if a map $S: \RR^n \to \RR^n$ satisfies the compatibility condition, \cref{eq:compatibility}, then
% \begin{equation*}
%     W_2(S_\sharp \temp, \temp) = \Vert F_\refe(S_\sharp \temp) - F_\refe(\temp) \Vert_\refe.
% \end{equation*}

%%% NOT SURE IF WE NEED TO STATE THIS %%%%%
% In general, however, \cite{aldroubi20} showed the following result (restated to conform to our notation):
% \begin{theorem}
% Let $n \geq 2$ and $\mathcal{H} = \{h: \RR^n \to \RR^n \vert h \text{ is differentiable}\}$ be a set of transformations.  If for some fixed reference distribution $\refe$ the compatibility condition $h \circ T_\refe^\temp = T_\refe^{h_\sharp \temp}$ holds for all $h \in \mathcal{H}$, then $\mathcal{H}$ only contains translations and isotropic scalings.
% \end{theorem}
% In particular, for a fixed reference $\refe$ and some transformation $S$, if we want compatibility to hold for all template distributions, then our transformation $S$ must have been a shift and/or isotropic scaling.  As the set of shifts and scalings is rather restrictive for applications, it is our aim to extend the set of functions $S$ that satisfy the compatibility condition, by understanding the relationship between reference distributions, template distributions, and $S$.

% Notice that shears of the form $Ax + b$ for a symmetric and positive definite matrix $A \in \RR^{n \times n}$ and a shift $b \in \RR^n$ are also part of $\mathcal{F}(P)$ if the spectral decomposition of $A$ has $P$ as the orthogonal matrix.

Given a fixed template distribution $\temp$, we show that demanding that the compatibility condition holds (under suitable conditions), if we fix either the reference distribution $\refe$ or the set of transformations, then the other (either the reference or transformations) can be fully characterized.

{\bf Fixed Reference and Template: }
Assume we fix the template distribution $\temp$ and reference distribution $\refe$.  If the Jacobian of $T_\refe^\temp(x)$ has spectral decomposition $P^\top D(x) P$ for a constant orthogonal matrix $P$, then the set of compatible transformations can be fully characterized: 

\begin{theorem}[Conditions on transformations]\label{converseShearTheorem}
Let $\refe, \temp \in \mathcal{P}_2(\RR^n)$ with $\refe \ll \lambda$. If the Jacobian of $T_\refe^\temp$ has a constant orthonormal basis given by an orthogonal matrix $P$ (i.e. $J_{T_\refe^\temp}(x) = P^\top D(x) P$), then $\mathcal{F}(P)$ is the set of transformations for which the compatibility condition (\Cref{def:compatibility}) holds.
\end{theorem}
\begin{proof}
The proof of the theorem can be found in \cref{ProofConverseShearTheorem}.
\end{proof}

%%%%%%%% COMMENTS FOR LATER %%%%%%%%%%%%%%%%%%
% \todo[inline]{Future: Fix $\refe,\temp$, without any assumptions on Jacobian. What can we say about the elementary transformations?}
% \todo[inline]{Future: If $\refe=\lambda$, does the Jacobi thing get simpler?}

\begin{example}[Gaussians]
To illustrate \Cref{converseShearTheorem}, we provide a simple example with Gaussians. If both $\refe$ and $\temp$ are Gaussian distributions, for example $\mathcal{N}(m_1,I)$ and $\mathcal{N}(m_2,\Sigma_2)$, then 
\begin{equation*}
    T_{\refe}^{\temp}(x) = m_2 + \Sigma_2(x-m_1), 
\end{equation*}
and $J_{T_{\refe}^{\temp}}(x)=\Sigma_2$.  If $\Sigma_2$ is positive definite, then it can be decomposed as $P^TDP$. Therefore, \Cref{converseShearTheorem} allows all generalized shears in \Cref{def:shears} that point in the same direction as $\Sigma_2$.
\end{example}

{\bf Fixed Shear and Template:} Now we fix the transformation to be a type of generalized shear and the template distribution $\temp$, and characterize the set of reference distributions such that compatibility condition holds.

\begin{theorem}[Conditions on reference distribution]\label{shearTheorem}
Let $P$ be an orthogonal matrix, let $S(x) = P^\top g(Px) + b$ for $g(z) = \begin{bmatrix}
g_1(z_1) & \dotsc & g_n(z_n) \end{bmatrix} : \RR^n \to \RR^n$ where $g_j: \RR \to \RR$ is differentiable and $b \in \RR^n$, and let $\temp \in \mathcal{P}_2(\RR^n)$ be a fixed template distribution with $\temp \ll \lambda$. Then $\Sigma = \{f_\sharp \temp : f \in \mathcal{F}(P) \}$ is the set of reference distributions such that the compatibility condition (\Cref{def:compatibility}) holds.
\end{theorem}
\begin{proof}
The proof can be found in \Cref{ProofshearTheorem}.
\end{proof}

In \Cref{shearTheorem}, note that the reference distributions in $\Sigma$ end up being absolutely continuous since they are the smooth pushforward of an absolutely continuous measure.  Additionally, we get the following corollary.

\begin{corollary}
Given the family of transformations of the form $S(x)$ from \Cref{shearTheorem} above, the set of reference distributions such that the compatibility condition holds for all of the transformations simultaneously is $\Sigma = \{f_\sharp \temp : f \in \mathcal{F}(P) \}$.
\end{corollary}
\begin{proof}
Inspecting the proof of \Cref{shearTheorem}, we see that the set of reference distributions $\Sigma$ does not depend on the choice of functions $g_j: \RR \to \RR$ but rather only on $P$.  
\end{proof}

A corollary of the theorems above is when the transformations used are constant shears.
\begin{corollary}
Consider an affine transformation $S(x) = Ax+b$, where $A$ is symmetric positive definite with orthonormal basis given by an orthogonal matrix $P$. For a template distribution $\temp \in \mathcal{P}_2(\RR^n)$ with $\temp \ll \lambda$, $\Sigma = \{f_\sharp \temp : f \in \mathcal{F}(P)\}$ is the set of reference distributions such that the compatibility condition holds.
\end{corollary}

\begin{example}[Gaussians with fixed shear]
To illustrate \Cref{shearTheorem}, we provide a simple example again with Gaussians.  Let $\temp = \mathcal{N}(m_1, I_n)$.  Consider a symmetric positive definite matrix $A$ with spectral decomposition $A = P^\top \Lambda P$ and a corresponding fixed shear $S(x) = A x + b$ for some $b \in \RR^n$, which yields the pushforward $S_\sharp \mu = \mathcal{N}(A m_1 + b, AA^\top)$.  For simplicity, we will check that the subset of compatible affine transformations 
\begin{align*}
    \mathcal{F}_{\text{affine}}(P) &= \{ f(x) = C x + d  : f \in \mathcal{F}(P)\} \\
    &= \{ P^\top D P x + d : D_{ij} = 0 \hspace{0.1cm}\forall\hspace{0.1cm} i \neq j,  D_{ii} > 0, d \in \RR^n \}
\end{align*}
yields reference distributions $\refe \in \{f_\sharp \mu : f \in \mathcal{F}_{\text{affine}}(P) \}$ so that the compatibility condition hold.  In particular note that for $f(x) = Cx + d = P^\top D P x + d$, our reference distributions have the form
\begin{align*}
    \refe = \mathcal{N}( Cm_1 + d, C C^\top) = \mathcal{N}( C m_1 + d, P^\top D^2 P).
\end{align*}
Since the optimal transport map between two general Gaussians $\mathcal{N}(\Tilde{m}_1, \Sigma_1) \to \mathcal{N}(\Tilde{m}_2, \Sigma_2)$ is given by
\begin{align*}
    \Tilde{m}_2 + \Sigma_1^{-\frac{1}{2}} ( \Sigma_1^{\frac{1}{2}} \Sigma_2 \Sigma_1^{\frac{1}{2}} )^{\frac{1}{2}} \Sigma_1^{-\frac{1}{2}} \Big(x - \Tilde{m}_1 \Big),
\end{align*}
see \cite{takatsu11}, we know that
\begin{align*}
    T_\refe^\temp &= m_1 + \underbrace{(CC^\top)^{-\frac{1}{2}} ( CC^\top)^{\frac{1}{2}} (CC^\top)^{-\frac{1}{2}}}_{(CC^\top)^{-1/2} = (C^2)^{-1/2} = C^{-1}} \Big( x - Cm_1 + d\Big) \\
    &= m_1 + C^{-1}( x - Cm_1 - d) = C^{-1} (x - d).
\end{align*}
So we have that
\begin{align*}
    S \circ T_\refe^\temp(x) = A C^{-1}(x -d) + b.
\end{align*}
On the other hand because $C = P^\top D P = C^\top$ and $A = P^\top \Lambda P = A^\top$ (so that $AC = CA$), we have that
\begin{align*}
    T_\refe^{S_\sharp \temp} &= Am_1 + b + M( x - (Cm_1 + d)) \\
    M &= (CC^\top)^{-1/2} ( (CC^\top)^{1/2} AA^\top (CC^\top)^{1/2} )^{1/2} (CC^\top)^{-1/2}\\
    &= C^{-1} ( C^{-1} A^2 C^{-1} )^{1/2} C^{-1} = C^{-1} (C^2 A^2)^{1/2} C^{-1} = A C^{-1} \\
    \implies T_\refe^{S_\sharp \temp}(x) &= \underbrace{Am_1 - Am_1}_{0} + b + AC^{-1} ( x - d) = S \circ T_\refe^\temp (x).
\end{align*}
So we actually get compatibility here and in Appendix section \ref{example11appendix} we present a numerical validation of this fact.
\label{example11}
\end{example}

% \begin{proof}
% Note that $S(x) = Ax + b$ is a special case of $P^\top g(Px)$, where $g_i(z) = \lambda_i z$ for eigenvalues $\lambda_i$.  The result follows from \cref{shearTheorem}.
% \end{proof}

{\bf Shears are Not Compatible in General:} Another consequence of \Cref{mainTheorem} is that non-trivial orthogonal transformations cannot be transformations that satisfy the compatibility condition.  
\begin{theorem}\label{orthogonalImpliesIdentity}  Let $\refe \ll \lambda, \temp \in \mathcal{P}_2(\RR^n)$, and let $S(x) = Ax + b$ be a compatible transformation (i.e $S \circ T_\refe^\temp = T_\refe^{S_\sharp \temp}$) such that $b \in \RR^n$ is a shift and $A \in \RR^{n\times n}$ is an orthogonal matrix.  Then $A$ must be the identity.
\end{theorem}

\begin{proof}
The proof can be found in \Cref{ProoforthogonalityImpliesIdentity}.
\end{proof}

% RECONSIDER THIS LATER
% \todo[inline, color=green]{Should I also insert the result I mentioned about orthogonal Hessians for the $\RR^2$ case?  -- >  Very specific subcase for generalization.  Only works in specific setting.  Maybe put in later.}

%%%%%%%%%% RESULT ON CLASSIFICATION %%%%%%%%%%%%%%%%%%%%%%%%

\section{Binary classification with pre-specified separation}\label{sec:delta_separation}

The main application of LOT isometries is to embed a subset of $\mathcal{P}_2(\RR^n)$ into a linear space where binary classification is easily accomplished via linear separability.  We show that data generated from a suitably bounded set of transformations still allows to have LOT linear separability in a suitable supervised learning paradigm.  We only try to classify two classes. 
Consider the following data-generating process:
\begin{definition}[Elementary Transformation Generated Process]
Consider a class of functions $\mathcal{H} = \{ h : \RR^n \to \RR^n  \}$.  Let $\temp_1$ or $\temp_2$ be two base probability measures.  Then we call $\mathcal{H} \star \temp_1 = \{h_\sharp \temp_1 : h \in \mathcal{H} \}$ and  $\mathcal{H} \star \temp_2 = \{h_\sharp \temp_2 : h \in \mathcal{H} \}$ the measures generated from elementary transformation $\mathcal{H}$ and $\temp_1$ and $\mathcal{H}$ and $\temp_2$, respectively.  Moreover, assume that $\mathcal{H}\star \temp_1$ have label $y = 1$ and $\mathcal{H}\star \temp_2$ have label $y=-1$.
\end{definition}

Given a reference $\refe$ and a set of measures $\mathcal{Q}$, let $F_\refe(\mathcal{Q})$ be the embedding of $\mathcal{Q}$ into the LOT space $L^2(\RR^n, \refe)$.  Given the data generating process above, our goal is to show that the linear separability of $F_\refe(\mathcal{H}\star \temp_1)$ and $F_\refe(\mathcal{H} \star \temp_2)$ is well characterizable with respect to $\mathcal{H}$ and the distance between $\temp_1$ and $\temp_2$.  We summarize the main result in the theorem below with proof given in \cref{separabilityProof}:

\begin{theorem}\label{separability}
Consider two distributions $\temp_1,\temp_2 \in \wspaceR$ with Wasserstein-2 distance $W_2(\temp_1, \temp_2)>0$ and assume that $\temp_1$ and $\temp_2$ have bounded support.  Pick a separation level $\delta$ such that $W_2(\temp_1, \temp_2) > \delta > 0$ and an error level $\epsilon > 0$.  Define $L \leq \frac{W_2(\temp_1,\temp_2) - \delta}{2} - \epsilon$.  Let
\begin{align*}
    \mathcal{H} = \{h: \RR^n \to \RR^n \vert h = \nabla \phi \text{ for convex $\phi$} \}
\end{align*}
be some generic convex set of transformations such that $\sup_{h \in \mathcal{H}} \Vert h - I \Vert_{\temp_1} \leq L$ and $\sup_{h \in \mathcal{H}} \Vert h - I \Vert_{\temp_2} \leq L$.  Furthermore, define the $\epsilon$-tube of this set of transformations
\begin{align*}
    \mathcal{H}_\epsilon = \{ \Tilde{h}: \RR^n \to \RR^n \vert \hspace{0.1cm} \Vert h - \Tilde{h} \Vert_{\temp_i} < \epsilon; i \in \{ 1, 2 \} \}.
\end{align*}
Then, for any choice of reference $\refe \in \mathcal{P}_2(\RR^n)$ that is absolutely continuous with respect to the Lebesgue measure, the sets  $F_{\refe}(\mathcal{H}_\epsilon \star \temp_1)$ and $F_{\refe}(\mathcal{H}_\epsilon \star \temp_2)$ are linearly separable with separation at least $\delta$.
\end{theorem}

\begin{remark}
The bounds on the function class $\mathcal{H}$ ensure that $\mathcal{H} \star \mu_1$ and $\mathcal{H} \star \mu_2$ are disjoint.  However, note that there can still exist function classes $\mathcal{H}$ without a bound on it, where $\mathcal{H} \star \mu_1$ and $\mathcal{H} \star \mu_2$ are still disjoint.  For example, you can consider the case when $\mathcal{H}$ is the set of all shifts and when $\mu_1$ and $\mu_2$ are a uniform distribution on the unit square and an isotropic Gaussian.  In this case, the sets $\mathcal{H} \star \mu_1$ and $\mathcal{H} \star \mu_2$ are disjoint.
\end{remark}

%\textcolor{red}{As stated, this is incorrect.  It's not just that $\mu_1$ and $\mu_2$ are different, but that the transformations $\mathcal{H}\star\temp_1$ and $\mathcal{H}\star\temp_2$ are disjoint.  For example, if $\temp_2$ is a shift of $\temp_1$, then they wouldn't be separable.}

\begin{remark}
Notice that the functions $\mathcal{F}(P)$ from \Cref{def:shears} satisfy the conditions of $\mathcal{H}$ for \Cref{separability} above.  In particular, every $S \in \mathcal{F}(P)$ can be written as $S = \nabla \phi$ for some convex $\phi$.  For this, let $p_{ij}$ denote the $(i,j)$th entry of $P^\top$, then we have that
\begin{align*}
    \phi(x) &= \int_\RR \Big( (S(x))_j \Big) dx_j = \int_\RR \sum_{k=1}^n p_{jk} f_k \Big( \sum_{i=1}^n p_{ki} x_i \Big) dx_j \\
    &= \sum_{k=1}^n p_{jk} \int_\RR f\bigg( \sum_{i=1}^n p_{ki} x_i \bigg) dx_j + C(x_1, \dotsc, x_{j-1}, x_{j+1}, \dotsc, x_n) .
\end{align*}
The positive definiteness and symmetry of $J_S$ implies that $\phi$ is convex.
\end{remark}

%%%%%%%%%%% FUTURE %%%%%%%%%%%%%%%%%%%%%%%%%
% \todo[inline,color=green]{Caroline:  Compare assumptions to Caroline and Alex's paper. Write a paragraph about no assumptions on compactness, but need bounded support.}
% \todo[inline,color=red]{Future: Make a version of this theorem where $\mathcal{H}$ is a transport up to an error (within some $\epsilon$-tube)  Need this for arxiv/journal.}

When we assume that $\mathcal{H}$ is compatible with respect to $\temp_1$ and $\temp_2$ and use either of these templates as the reference distribution, we actually gain better results than the general separation theorem above.  The proof for the theorem below is in \cref{compat_sep_proof}:

\begin{theorem}\label{compat_sep}
Fix $\temp_1, \temp_2 \in \mathcal{P}_2(\RR^n)$ with finite support and $\temp_1,\temp_2 \ll \lambda$, and let $\mathcal{H}$ be a convex set of transformations that are compatible with $\temp_1$ and $\temp_2$ (this includes shifts and scalings).  Let $\mathcal{H}_\epsilon = \{h_\epsilon : \Vert h - h_\epsilon \Vert_{\temp_j} < \epsilon, j = 1,2 \}$.
\begin{enumerate}
    \item (Linear separability) If $\mathcal{H} \star \temp_1$ and $\mathcal{H} \star \temp_2$ are disjoint, then $F_{\temp_1}(\mathcal{H} \star \temp_1 )$ and $F_{\temp_1}(\mathcal{H} \star \temp_2)$ are linearly separable.
    
    \item (Linear separability of $\epsilon$-tube functions) If the minimal separation between $\mathcal{H} \star \temp_1$ and $\mathcal{H} \star \temp_2$ is greater than $2\epsilon$, then $F_{\temp_1}(\mathcal{H}_\epsilon \star \temp_1)$ and $F_{\temp_1}(\mathcal{H}_\epsilon \star \temp_2)$ are linearly separable.
    
    \item (Sufficient conditions for separation) If we assume:
    \begin{enumerate}
        \item For every $h \in \mathcal{H}$ and every $x \in \RR^n$ that $\Vert h(x) \Vert_2 \geq \sqrt{2} \Vert x - x_0 \Vert_2$ where is the mean of the normalized measure $\vert \temp_1 -\temp_2 \vert$
    \begin{align*}
        x_0 = \frac{1}{\vert \temp_1 - \temp_2 \vert(\RR^n)} \int_{\RR^n} z d\vert \temp_1 - \temp_2\vert(z),
    \end{align*}
    
    \item $\sup_{h,\Tilde{h} \in \mathcal{H}} \Vert h - \Tilde{h}\Vert_{\temp_1} \leq W_2(\temp_1,\temp_2) - \delta - 2\epsilon$ for $\delta > 0$,
    \end{enumerate}
    then $F_{\temp_1}(\mathcal{H}_\epsilon \star \temp_1)$ and $F_{\temp_1}(\mathcal{H}_\epsilon \star \temp_2)$ are separated by at least $\delta >0$.
\end{enumerate}
\end{theorem}

\begin{remark}
Notice that if we choose $\mathcal{H}$ to be shifts and scalings, the first statement of \Cref{compat_sep} is the direct generalization of corollary 4.3 of \cite{moosmueller20} since shifts and scalings are compatible with every probability measure.
\end{remark}

\begin{remark}
Notice that in \Cref{compat_sep}, the condition $W_2(\temp_1, \temp_2) -\delta \geq \sup_{h,\Tilde{h} \in \mathcal{H}} \Vert \Tilde{h} - h \Vert_{\temp_1}$ in the third statement is essentially the same condition the one in \Cref{separability} because by rewriting the condition in \Cref{separability}, we get $\sup_{h \in \mathcal{H}} \Vert h - I \Vert_{\temp_1} \leq \frac{W_2(\temp_1, \temp_2) - \delta}{2}$.  This comes from the fact that 
\begin{align*}
    2 \sup_{h} \Vert h - I \Vert_{\temp_1} \geq \sup_{h,\Tilde{h} \in \mathcal{H}} \Vert \Tilde{h} - h \Vert_{\temp_1} \geq \sup_{h \in \mathcal{H}} \Vert  h - I \Vert_{\temp_1} - \inf_{\Tilde{h} \in \mathcal{H}} \Vert \Tilde{h} - I \Vert_{\temp_1}.
\end{align*}
If the problem setting allows $I \in \mathcal{H}$, then the right hand side is just $\sup_{h \in \mathcal{H}} \Vert h - I \Vert_{\temp_1}$.  Thus, in this case, \Cref{compat_sep} is stronger than \Cref{separability} since our function class has the larger bound $\sup_{h \in \mathcal{H}} \Vert h - I \Vert_{\temp_1} \leq  W_2(\temp_1,\temp_2) - \delta$.
\end{remark}

\Cref{separability} above acts as a blueprint for controlling the degree of separation in the LOT embedding via the bounds of the function class $\mathcal{H}$.  For the specific setting of shears,
\begin{equation}\label{shearsH}
    \mathcal{H}_{\gamma, M, M_b} = \Big\{ Ax + b : \substack{\text{$A$ is symmetric positive definite with} \\ \text{$\lambda_{\text{min}}(A) > \gamma$ and $\lambda_{\text{max}}(A) < M$, and $\Vert b \Vert_2 \leq M_b$} } \Big\},
\end{equation}
we can choose $\gamma, M,$ and $M_b$ in a way that guarantees that $F_{\refe}(\mathcal{H}_{\gamma, M, M_b} \star \temp_1)$ and $F_{\refe}(\mathcal{H}_{\gamma, M, M_b} \star \temp_2)$ are $\delta$-separated.  This leads us to the following corollary with proof in \cref{ProofshearSeparability}.

%%%%%%%%%%%%%% FUTURE %%%%%%%%%%%%%%%%%%%%%%%%%%%%%%%%%
% \todo[inline]{For the shearing, try to make more informal.  Something like "if you have these bounds on the eigenvalues, then shears work." Make a corollary without shifts and then put the corollary below into the appendix.}

\begin{corollary}\label{shearSeparability}
Consider two distributions $\temp_1$ and $\temp_2$ with Wasserstein-2 distance $W_2(\temp_1, \temp_2)$.  Let us denote $R_1 = \max_{x \in \text{supp}(\temp_1)} \Vert x \Vert_2$ and $R_2 = \max_{x \in \text{supp}(\temp_2)} \Vert x \Vert_2$.  For the function class of shears $\mathcal{H}_{\gamma, M, 0}$ and $\refe \ll \lambda$, we can ensure that $F_\refe(\mathcal{H}_{\gamma, M, 0} \star \temp_1)$ and $F_\refe(\mathcal{H}_{\gamma, M, 0} \star \temp_2)$ are $\delta$-separated if 

\textbf{Case 1:}  assuming that $W_2(\temp_1, \temp_2) > (R_1 + R_2) + \delta$, then $M$ is chosen such that
\begin{align*}
    2 < M \leq \frac{W_2(\temp_1, \temp_2) - \delta + (R_1+R_2)}{R_1 + R_2},
\end{align*}
and

\textbf{Case 2:} assuming that $\delta < W_2(\temp_1, \temp_2) < (R_1 + R_2)$, then $M$ and $\gamma$ is chosen such that
\begin{align*}
    1 < M \leq \frac{W_2(\temp_1, \temp_2) - \delta + (R_1+R_2)}{R_1 + R_2} \\
    \gamma \geq \frac{\delta - W_2(\temp_1, \temp_2) + R_1 + R_2}{R_1 + R_2}.
\end{align*}
\label{ShearSeparableCorollary}
\end{corollary}
\begin{proof}
This comes straight from \Cref{shearSeparabilityGeneral} provided that $b = 0$ and $\epsilon = 0$.
\end{proof}

\section{Binary Classification with Multiple References}\label{sec:multiple_ref}
It is possible to achieve better separation with a larger function class than the class of bounded shears described in Section \ref{sec:delta_separation}.
%Although we can connect the degree of linear separability to the boundedness of our function class $\mathcal{H}$, we can actually achieve better separation with a larger function class. 
The cost of this better separation, however, is to use multiple LOT spaces.  Note that once a set of two measures $\mathcal{H} \star \mu_1$ and $\mathcal{H} \star \mu_2$ are separable in LOT space with respect to one reference (from \Cref{separability}), then $\mathcal{H} \star \mu_1$ and $\mathcal{H} \star \mu_2$ must be separable in LOT space with respect to multiple references.

%%%%%%%% FUTURE ? OR DONE %%%%%%%%%%%%%%%%%%%%%%%%%%%%%%%%%%%%%%%%%%%%
% \todo[inline]{Make the trivial statement that when we are separable with one reference, then we are separable with multiple references (increase embedding space) (done)}

\begin{lemma}\label{multipleRef}
Let $\mu, \gamma \in \mathcal{P}_2$, $\epsilon > 0$, and
\begin{align*}
    \mathcal{H} = \{h: \RR^n \to \RR^n \vert h = \nabla \phi \text{ for convex $\phi$} \}
\end{align*}
be such that $\sup_{h \in \mathcal{H}} \Vert h - I \Vert_\mu \leq L$ and $\sup_{h \in \mathcal{H}} \Vert h - I \Vert_\gamma \leq L$, where $2(L+\epsilon) < W_2(\mu,\gamma)$.  Consider a desired separation level $\delta^*$.  If we have absolutely continuous (with respect to the Lebesgue measure) reference measures $\refe_1, \dotsc, \refe_N$ for $N \geq \Big(\frac{\delta^\star}{W_2(\mu, \gamma) - 2(L+\epsilon)}\Big)^2$, then
\begin{align*}
    F_N(\mu) = F_{\refe_1}(\mathcal{H}_\epsilon \star \mu) \times \dotsc \times F_{\refe_N}(\mathcal{H}_\epsilon \star \mu) \\
\end{align*}
and
\begin{align*}
    F_N(\gamma) = F_{\refe_1}(\mathcal{H}_\epsilon \star \gamma) \times \dotsc \times F_{\refe_N}(\mathcal{H}_\epsilon \star \gamma)
\end{align*}
are $\delta^*$-separable.
\end{lemma}

Notice that the \Cref{multipleRef} allows one to pick a larger function class $\mathcal{H}$ and a small separation level $\delta^*$; however, the number of LOT spaces that you must embed into is the cost of this better performance.

As preliminaries of proving this result, let's discuss the product metric and some general results for normed spaces that will give us intuition in our analysis.  Given the metrics $d_{\refe_1}$ and $d_{\refe_2}$ for $F_{\refe_1}(\mathcal{P}_2)$ and $F_{\refe_2}(\mathcal{P}_2)$, respectively, and $(T_1,S_1), (T_2, S_2) \in F_{\refe_1}(\mathcal{P}_2) \times F_{\refe_2}(\mathcal{P}_2)$, the $\ell_p$ product metric on $F_{\refe_1}(\mathcal{P}_2) \times F_{\refe_2}(\mathcal{P}_2)$ is 
\begin{align*}
    d_{p, \refe_1 \times \refe_2}\big( (T_1,S_1), (T_2, S_2) \big) = \Vert \big( d_{\refe_1}(T_1, T_2), d_{\refe_2}(S_1,S_2) \big) \Vert_p.
\end{align*}
In particular, this $\ell_p$ product metric is just the regular $\ell_p$ Euclidean norm applied to the Euclidean point $x = \big( d_{\refe_1}(T_1, T_2), d_{\refe_2}(S_1,S_2) \big)$.  Moreover, note that we can easily extend this definition to a product space of more than 2 spaces.

A basic (well-known) exercise in linear algebra shows that in any finite dimensional vector space $V$, for any $0 < r < p$, and for $x \in V$, we have
\begin{align*}
    \Vert x \Vert_p \leq \Vert x \Vert_r \leq n^{\frac{1}{r}-\frac{1}{p}} \Vert x \Vert_p.
\end{align*}
Even though $F_{\refe_1}(\mathcal{P}_2) \times F_{\refe_2}(\mathcal{P}_2)$ is an infinite-dimensional space, the product metric on this product space is actually acting on $\RR_{>0} \times \RR_{>0}$.  This means that the $\ell_p$ and $\ell_r$ norm inequalities above hold for our product space when endowed with the product metric.  This essentially signals ``stronger" linear separability.

To see this, assume that $F_{\refe_1}(\mathcal{H} \star \mu)$ and $F_{\refe_1}(\mathcal{H} \star \gamma)$ are $\delta_1$-separated in $F_{\refe_1}(\mathcal{P}_2)$ and that $F_{\refe_2}(\mathcal{H} \star \mu)$ and $F_{\refe_2}(\mathcal{H} \star \gamma)$ are $\delta_2$-separated, then in the product space, we have
\begin{align*}
    \max(\delta_1,\delta_2) &= \Bigg\Vert \begin{pmatrix}
    d_{\refe_1}(F_{\refe_1}(\mathcal{H} \star \mu), F_{\refe_1}(\mathcal{H} \star \gamma)) \\
    d_{\refe_2}( F_{\refe_2}(\mathcal{H} \star \mu) , F_{\refe_2}(\mathcal{H} \star \gamma))
    \end{pmatrix}  \Bigg\Vert_\infty \\
    &\leq \Bigg\Vert  \begin{pmatrix}
    d_{\refe_1}(F_{\refe_1}(\mathcal{H} \star \mu), F_{\refe_1}(\mathcal{H} \star \gamma)) \\
    d_{\refe_2}( F_{\refe_2}(\mathcal{H} \star \mu) , F_{\refe_2}(\mathcal{H} \star \gamma))
    \end{pmatrix} \Bigg\Vert_2 \\
    &\leq \sqrt{2} \Bigg\Vert \begin{pmatrix}
    d_{\refe_1}(F_{\refe_1}(\mathcal{H} \star \mu), F_{\refe_1}(\mathcal{H} \star \gamma)) \\
    d_{\refe_2}( F_{\refe_2}(\mathcal{H} \star \mu) , F_{\refe_2}(\mathcal{H} \star \gamma))
    \end{pmatrix}  \Bigg\Vert_\infty \\
    &= \sqrt{2} \max(\delta_1, \delta_2).
\end{align*}
We are more interested, however, in providing lower bounds for the product $\ell_2$-norm.  To investigate this, let's assume that $\mathcal{H}$ is fixed and that we have $N$ templates distributions $\refe_1, \dotsc, \refe_N$.  Now if $\mu$ is a generic distribution, let
\begin{align*}
    F_N(\mu) &= F_{\refe_1}(\mathcal{H} \star \mu) \times F_{\refe_2}(\mathcal{H} \star \mu) \times \dotsc \times
    F_{\refe_N}(\mathcal{H} \star \mu) \\
    &\subseteq F_{\refe_1}(\mathcal{P}_2) \times \dotsc \times F_{\refe_N}(\mathcal{P}_2)
\end{align*}
denote the embedding of $\mathcal{H} \star \mu$ into the product LOT space defined by $\refe_1, \dotsc, \refe_N$.  We will now prove the result.

\begin{proof}[Proof of \cref{multipleRef}]
From \cref{separability}, we know that for every $j$, $F_{\refe_j}(\mathcal{H} \star \mu)$ and $F_{\refe_j}(\mathcal{H} \star \mu)$ can be $\delta_j$-separated for some $\delta_j < W_2(\mu,\gamma)-2(L+\epsilon)$, where $\delta_j$ will be determined later.  Now notice that the degree of separation in the product space is
\begin{align*}
    \Bigg\Vert  \begin{pmatrix}
    \delta_1 \\
    \vdots \\
    \delta_N
    \end{pmatrix} \Bigg\Vert_2 = \sqrt{\sum_{j=1}^N \delta_j^2} < \sqrt{\sum_{j=1}^N (W_2(\mu,\gamma) - 2(L+\epsilon))^2} = \sqrt{N} (W_2(\mu, \gamma) - 2(L+\epsilon)).
\end{align*}
Thus, if we want to be at least $\delta^*$-separated in the product space, then we must have
\begin{align*}
    \delta^* \leq \Bigg\Vert  \begin{pmatrix}
    \delta_1 \\
    \vdots \\
    \delta_N
    \end{pmatrix} \Bigg\Vert_2 <
    \sqrt{N} (W_2(\mu, \gamma) - 2(L+\epsilon)) \\
    \implies N > \bigg( \frac{\delta^*}{W_2(\mu,\gamma) - 2(L+\epsilon)} \bigg)^2
\end{align*}
So we're done.
\end{proof}

\begin{example}\label{multipleLOT_Ex}
To show the tradeoff of \Cref{multipleRef}, let's try a multiple LOT embedding example with Gaussians.  Using the previous examples, assume that we have two template distributions $\mu_1 = \mathcal{N}(0, \Sigma_1)$ and $\mu_2 = \mathcal{N}(0, \Sigma_2)$.  We know that $W_2(\temp_1,\temp_2)^2 = \Tr( \Sigma_1 + \Sigma_2 - 2(\Sigma_1^{\frac{1}{2}} \Sigma_2 \Sigma_1^{\frac{1}{2}} )^{1/2} )$.  We consider the set of shears 
\begin{align*}
    \mathcal{H} = \{ A x : A =A^\top \in \RR^{n\times n}, M I_n \succeq A \succeq m I_n \succ 0 \}
\end{align*}
as our set of transformations, and to ensure separation, we use $L \leq \frac{W_2(\temp_1, \temp_2) - \delta}{2}$, which is shown in \Cref{multipleEx} to imply that
\begin{align*}
    \max \big( \vert M - 1\vert , \vert 1 - m \vert \big) \leq \frac{ W_2(\temp_1,\temp_2) - \delta }{2\max_{j = 1,2} \Vert \Sigma_j^{1/2} \Vert_F}.
\end{align*}
Now let us define our reference distributions to be of the form $\refe_1 = (h_1)_\sharp \temp_1$ and $\refe_2 = (h_2)_\sharp \temp_2$ for $h_1(x) = A_1 x$ and $h_2(x) = A_2 x$ for $h_1, h_2 \in \mathcal{H}$ so that
\begin{align*}
    \refe_1 = (h_1)_\sharp \mu_1 = \mathcal{N}( 0, A_1 \Sigma_1 A_1^\top), \hspace{0.4cm} \refe_2 = (h_2)_\sharp \mu_2 = \mathcal{N}( 0, A_2 \Sigma_2 A_2^\top).
\end{align*}
Notice that the bounds on $M$ and $m$ imply that there are infinite choices of reference distributions to choose from.  Moreover, we show in \Cref{multipleEx} that
\begin{align*}
    \frac{M^2}{m} W_2(\temp_1,\temp_2) \geq \Vert T_{\sigma_j}^{h_\sharp \mu_1} - T_{\sigma_j}^{ \Tilde{h}_\sharp \mu_2} \Vert_{\sigma_j} \geq \frac{m^2}{M} W_2(\temp_1,\temp_2)
\end{align*}
for our choices of reference distributions.  Now choosing $N$ reference distributions, our multiple LOT embedding has minimal separation bounded below by \begin{align*}
    \sqrt{ \sum_{j=1}^N \Vert T_{\sigma_j}^{(h_1)_\sharp \mu_1} - T_{\sigma_j}^{ (h_2)_\sharp \mu_2} \Vert_{\sigma_j}^2 } &\geq \sqrt{ \sum_{j=1}^N \frac{m^4}{M^2} W_2(\temp_1,\temp_2)^2 } = \sqrt{N} \frac{m^2}{M} W_2(\temp_1,\temp_2).
\end{align*}
Notice that as $\delta$ becomes closer to $W_2(\temp_1,\temp_2)$, we find that both $m$ and $M$ become closer to 1, which means that our set of shears become closer to the identity.  Using multiple LOT embeddings; however, we can actually use the maximal function class of shears $\mathcal{H}$ when $M = 1 + \frac{ W_2(\temp_1, \temp_2) }{2 \max_{j=1,2} \Vert \Sigma_j^{1/2} \Vert_F }$ and $m = 1 - \frac{ W_2(\temp_1, \temp_2) }{2 \max_{j=1,2} \Vert \Sigma_j^{1/2} \Vert_F }$.  To get the same separation with the largest possible function class as when we have $\delta > 0$, we need
\begin{align*}
    \sqrt{N} \Bigg( \frac{\Big(1 - \frac{ W_2(\temp_1, \temp_2) }{2 \max_{j=1,2} \Vert \Sigma_j^{1/2} \Vert_F }\Big)^2}{1 + \frac{ W_2(\temp_1, \temp_2) }{2 \max_{j=1,2} \Vert \Sigma_j^{1/2} \Vert_F }} \Bigg) &\geq \Bigg( \frac{\Big(1 - \frac{ W_2(\temp_1, \temp_2) - \delta }{2 \max_{j=1,2} \Vert \Sigma_j^{1/2} \Vert_F }\Big)^2}{1 + \frac{ W_2(\temp_1, \temp_2) - \delta }{2 \max_{j=1,2} \Vert \Sigma_j^{1/2} \Vert_F }} \Bigg).
\end{align*}
Rearranging the inequality and squaring both sides, we get the following bound for $N$
\begin{align*}
    N &\geq \Bigg(\frac{2 \max\limits_{j=1,2} \Vert \Sigma_j^{1/2} \Vert_F +   W_2(\temp_1, \temp_2)  }{ 2 \max\limits_{j=1,2} \Vert \Sigma_j^{1/2} \Vert_F +   W_2(\temp_1, \temp_2) - \delta  }\Bigg)^2  \Bigg( \frac{2\max\limits_{j=1,2} \Vert \Sigma_j^{1/2} \Vert_F - W_2(\temp_1, \temp_2) + \delta }{2\max\limits_{j=1,2} \Vert \Sigma_j^{1/2} \Vert_F - W_2(\temp_1, \temp_2)} \Bigg)^4.
\end{align*}
Thus, if needed, we can allow $\delta$ to stay small (or even become zero), which would allow us to use the maximal function class of shears $\mathcal{H}$; however, the cost of this larger function class and separation level is increasing the number of reference distributions.

\end{example}

%%%%%%%%%%%%%% Experiments %%%%%%%%%%%%%%%%%%%%%%%%%

\section{Numerical experiments}
\subsection{Binary classication of MNIST Images}
In this section we present pairwise binary classification results on sheared MNIST images which are motivated by the linear separability result presented in Corollary \ref{ShearSeparableCorollary} and also illustrate the benefit of using multiple references as indicated by lemma \ref{multipleRef}.

\subsubsection*{The LOT embedding pipeline for an image}
\label{pipeline}
\begin{enumerate}
    \item Obtain the image represented as a $n \times n$ matrix of pixel values.
    \item Assuming that the image is supported on a $n \times n$ grid on the unit square, obtain the point cloud which forms the support of the pixel values corresponding to the image.
    \item Obtain the discrete measure $\mu$ induced by the image on the unit square. Each point in the support of the image has a pixel value which (after normalization) will be the mass associated with $\mu$ .
    \item Let $\sigma$ denote a discrete reference measure \footnote{In case the desired reference is an absolutely continuous measure on the unit square, then we work with the discrete measure it induces on the $n \times n$ grid on the unit square (See Figure \ref{GaussianRefs}).}. Compute the discrete transport coupling matrix $P_{\sigma}^{\mu}$ \footnote{\href{https://pythonot.github.io/}{https://pythonot.github.io/} \cite{flamary2021pot}}. For each point $x$ in the support of  the reference $\sigma$, choose $T_{\sigma}^{\mu}(x)$ as the point in the support of $\mu$ such that $T_{\sigma}^{\mu}(x) = \text{argmax}_{y\in supp(\mu)} P_{\sigma}^{\mu}(x,y)$. Here $P(x,y)$ denotes the amount of mass transported from $x \in supp(\sigma)$ to $y \in supp(\mu)$. This is done to extract an approximate Monge map from the coupling matrix \cite{moosmueller20}.
    \item The LOT embedding of the image corresponding to the reference $\sigma$ is chosen to be $T_{\mu}^{\sigma}$. Note that $T_{\mu}^{\sigma} \in \mathbb{R}^{2m}$, where $m $ denotes the size of the size of the support $\sigma$, i.e. $m:= |supp(\sigma)|$. Henceforth this $\mathbb{R}^{2m}$ vector will be referred to as the \textit{LOT feature} corresponding to the particular image that is being embedded.
\end{enumerate}

\begin{figure}[H]
    \centering
    \includegraphics[width=\textwidth]{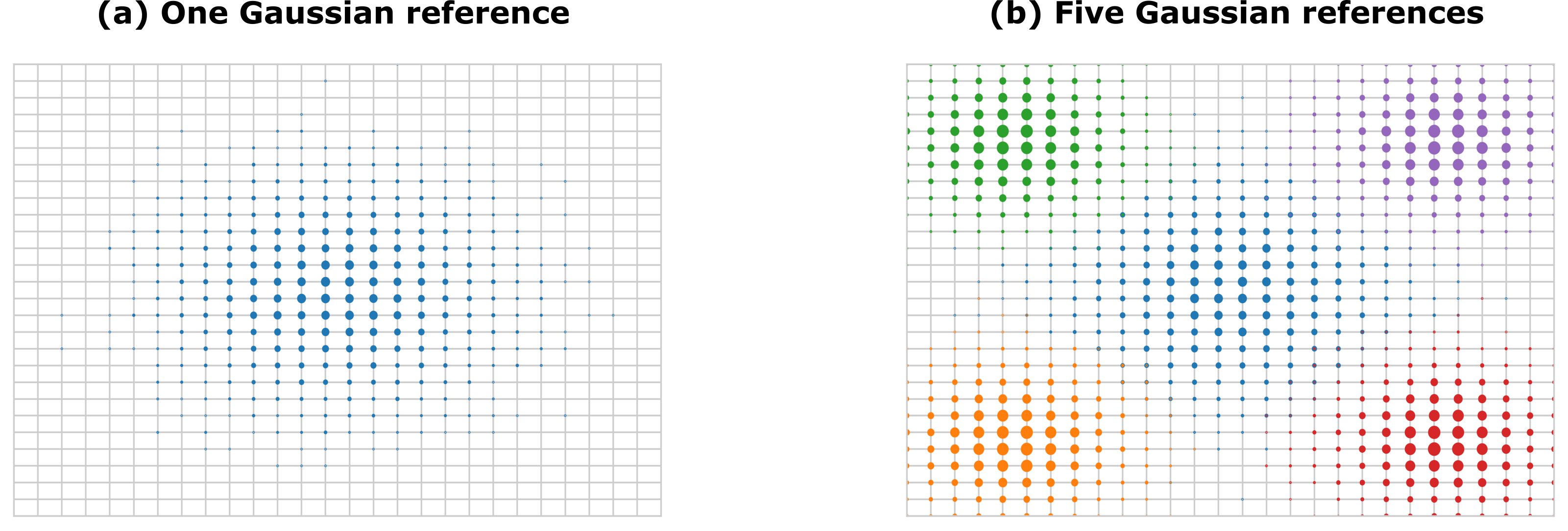}
    \caption{\scriptsize{a) A  Gaussian reference distribution approximated on a $28 \times 28$ grid}. b) Five different Gaussian distributions approximated on a $28 \times 28$ grid to be used as multiple reference for LOT embedding.}
    \label{GaussianRefs}
\end{figure}

\subsection{Experimental settings}
The MNIST images are sheared using the transformation described in Appendix \ref{shearing} and the values for each of the parameters $\lambda_1,\lambda_2,\theta, b$ are drawn randomly from a pre-fixed range for each image. We perform classification experiments for the MNIST images under two different shearing conditions (See Figure \ref{ShearedImages}). For one set of shearing conditions, termed as \textit{mild shearing} , the parameters of shearing for each image, $\lambda_1, \lambda_2$ are randomly chosen in the interval $[0.5,1.5]$, $\theta$ is randomly chosen in the interval $[0,360]$ degrees and the shifts $b$ are randomly chosen in the interval  $[-5,5]$. For the other set of shearing conditions termed as \textit{severe shearing}, the parameters of shearing for each image, $\lambda_1, \lambda_2$ are randomly chosen in the interval $[0.5,2.5]$, $\theta$ is randomly chosen in the interval $[0,360]$ degrees and the shifts $b$ are randomly chosen in the interval  $[-5,5]$. Then the \textit{LOT feature} corresponding to each of the sheared images are computed using the embedding pipeline described in subsection \ref{pipeline} and then classification experiments are performed using Linear Discriminant Analysis (LDA) \cite{hastie2009elements} \footnote{\href{https://scikit-learn.org/}{https://scikit-learn.org/}}.\\

\begin{figure}[H]
    \centering
    \includegraphics[width=\textwidth]{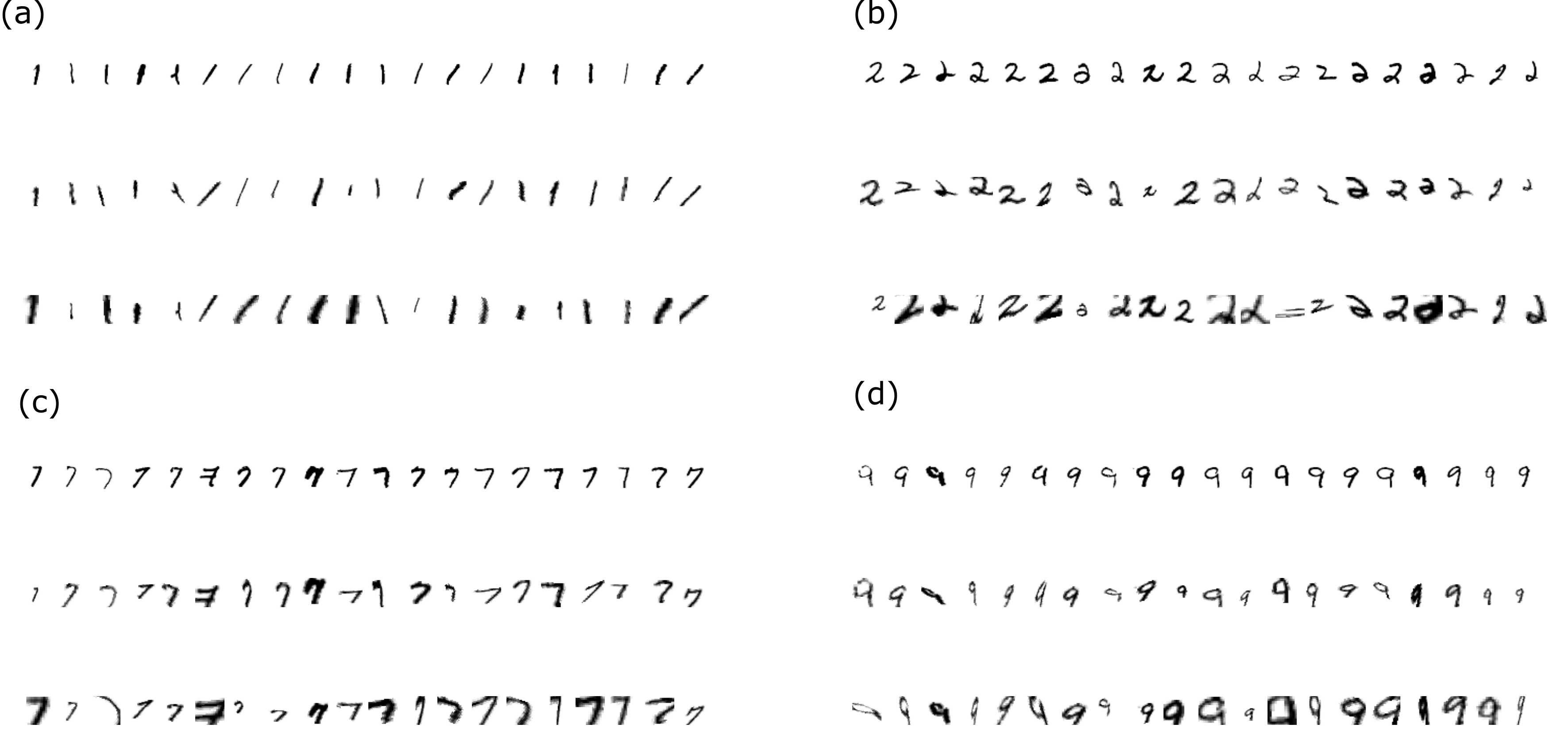}
    \caption{\scriptsize{In each figure, the first row shows the true unsheared MNIST image. The second row shows the corresponding mildly sheared MNIST image. The parameters (Appendix \ref{shearing}) of shearing for each image, $\lambda_1, \lambda_2$ are randomly chosen in the interval $[0.5,1.5]$, $\theta$ is randomly chosen in the interval $[0,360]$ degrees and the shifts $b$ are randomly chosen in the interval  $[-5,5]$. The third row shows the corresponding severely sheared MNIST image. The parameters (Appendix \ref{shearing}) of shearing for each image, $\lambda_1, \lambda_2$ are randomly chosen in the interval $[0.5,2.5]$, $\theta$ is randomly chosen in the interval $[0,360]$ degrees and the shifts $b$ are randomly chosen in the interval  $[-5,5]$}}
    \label{ShearedImages}
\end{figure}

To test the performance of LDA (Linear Discriminant Analysis) classification of two distinct classes of MNIST digits using LOT features, we study the test error of the LDA classifier as a function of the number of training images chosen for each digit. For each fixed number, $N_{train}$, of training images, we train the LDA classifier using a randomly chosen set of $N_{train}$  images from each digit class and test the classification results on a randomly chosen set of $1000$ test images from each digit class. We then repeat this experiment for each fixed $N_{train}$ using 20 different randomly chosen set of training images ($N_{train}$ images from each digit class) and $1000$ test images from each digit class. 

\subsection{Observations}
In Figure \ref{multirefImage12} we report the mean test error for classification of MNIST ones and twos and in Figure \ref{multirefImage79} we report the mean test error for classification of MNIST sevens and nines for various choices of reference distributions and under different shearing conditions. Therein for comparison, we also report the results obtained using the semi-discrete optimal transport \cite{merigot20} framework which uses a uniform reference measure. The corresponding standard deviations are reported in Appendix Figures \ref{multirefImage12std} and \ref{multirefImage79std}. We observe that the LOT framework is able to achieve low test errors with a relatively low number of training images. Moreover we see that using multiple references does indeed lead to a decrease in the classification error. Interestingly, we observe that using multiple references also helps reduce over-fitting (See Figure \ref{overfitting}). The trade-off observed is that using multiple references increases the length of the feature vector while on the other hand it leads to a decrease in the test error. 

In Figure \ref{heatmap} we illustrate as a heat-map, the mean test errors for binary classification of all pairs of MNIST digits using $50$ training images per class and for different choices of references. Also, in Table \ref{heatMapTable} we report the range of test errors and standard deviations observed across all the classification experiments corresponding to Figure \ref{heatmap}. Further in Appendix Figure \ref{CNNcomparison}, for comparison, we report the classification results for sheared MNIST 7s and 9s using convolutional neural networks with 1586 training parameters (labelled small CNN) and 3650 training parameters (labelled large CNN) under identical training and testing conditions as that of the discrete LOT classifier.
 
\begin{figure}[H]
    \centering
    \includegraphics[width=\textwidth]{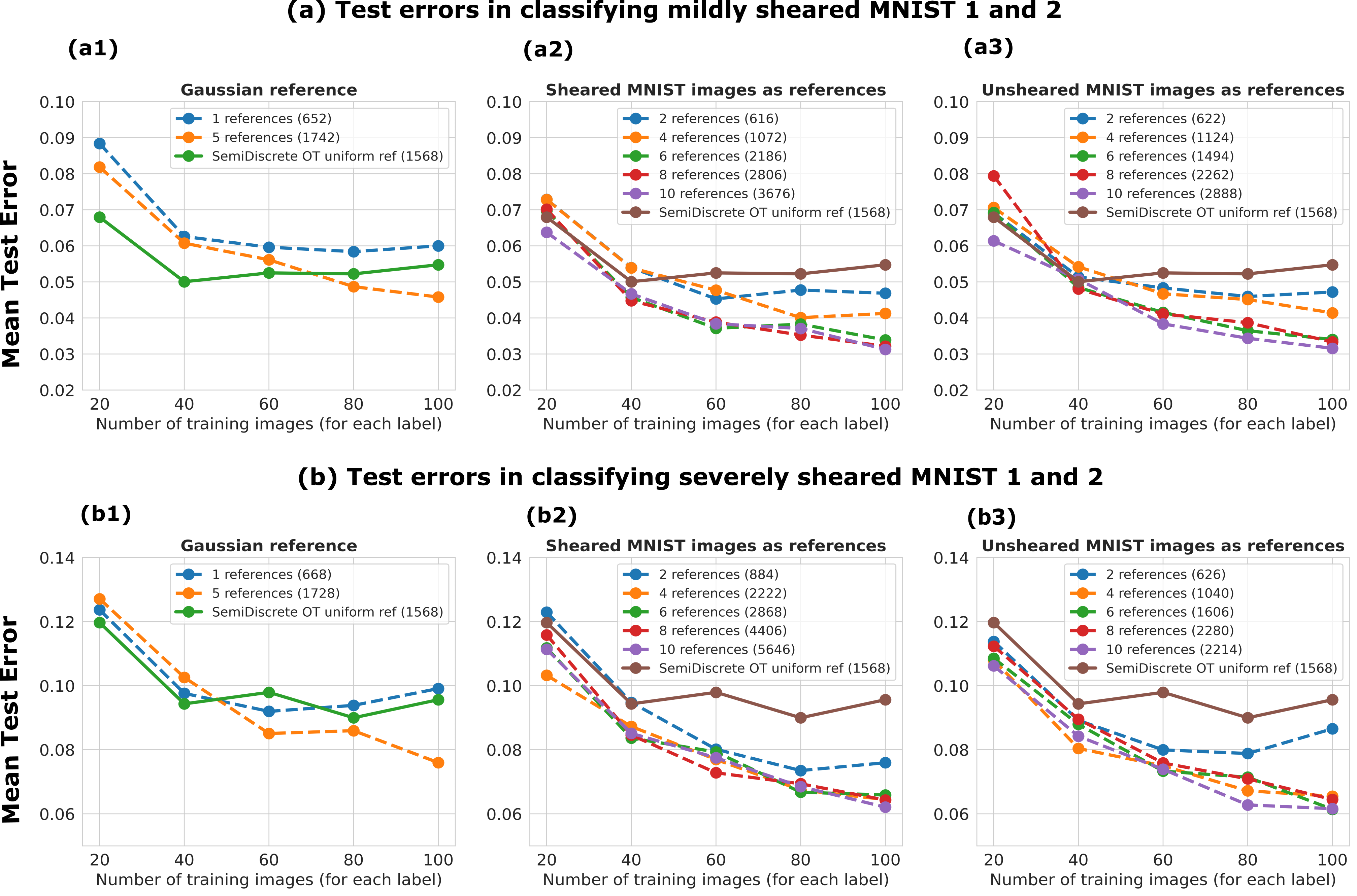}
    \caption{\scriptsize{(a) Test errors for binary classification of mildly sheared MNIST 1s and 2s using (a1) Gaussian references (a2) sheared MNIST 1s and 2s as references (a3) unsheared MNIST 1s and 2s as references. (b) Test errors for binary classification of severely sheared MNIST 1s and 2s using (b1) Gaussian references (b2) sheared MNIST 1s and 2s as references (b3) unsheared MNIST 1s and 2s as references. In the cases where MNIST images are used as references, the results are reported for the cases where the number of references used is $2i$ for $i=1, \cdots 5$ wherein $i$ images from each class are randomly drawn to be used as references from a pool of images that do not correspond to any of the training and testing images. For each fixed number of training images per class, $N_{train}$, the mean test classification error averaged across 20 random choices of $N_{train}$ training images (per class) and $1000$ test images (per class) is reported. The number inside the parenthesis in the legends of the images denote the length of the LOT feature vector corresponding to the particular choice of references. In all figures, for comparison, the results for classification using the semi discrete linear optimal transport framework \cite{merigot20} which uses the uniform measure as the reference is also reported. Standard deviations for each of the corresponding classification tests are reported in the Appendix Figure \ref{multirefImage12std}.}}
    \label{multirefImage12}
\end{figure}

\begin{figure}[H]
    \centering
    \includegraphics[width=\textwidth]{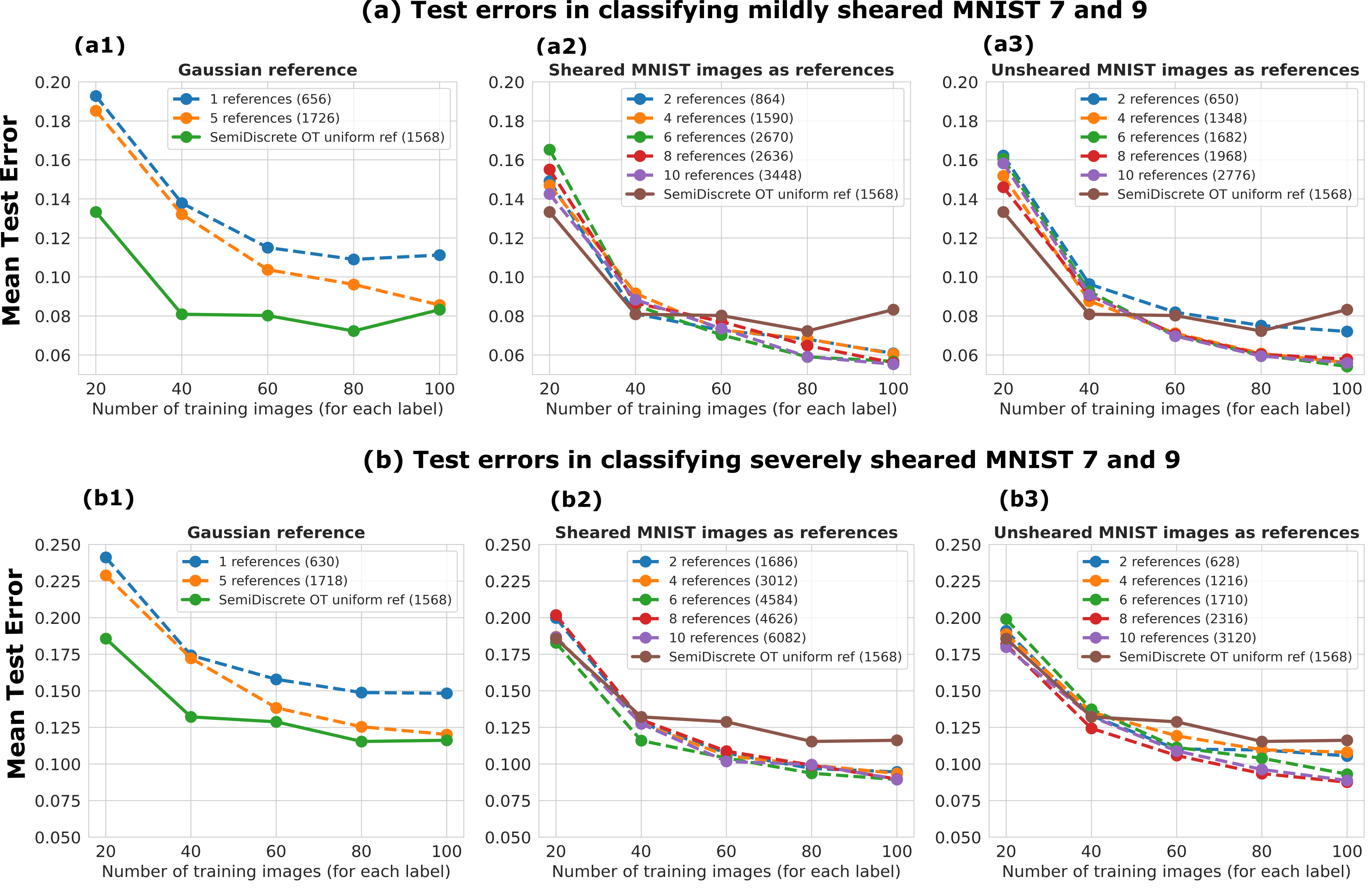}
    \caption{\scriptsize{(a) Test errors for binary classification of mildly sheared MNIST 7s and 9s using (a1) Gaussian references (a2) sheared MNIST 7s and 9s as references (a3) unsheared MNIST 7s and 9s as references. (b) Test errors for binary classification of severely sheared MNIST 7s and 9s using (b1) Gaussian references (b2) sheared MNIST 7s and 9s as references (b3) unsheared MNIST 7s and 9s as references. In the cases where MNIST images are used as references, the results are reported for the cases where the number of references used is $2i$ for $i=1, \cdots 5$ wherein $i$ images from each class are randomly drawn to be used as references from a pool of images that do not correspond to any of the training and testing images. For each fixed number of training images per class, $N_{train}$, the mean test classification error averaged across 20 random choices of $N_{train}$ training images (per class) and $1000$ test images (per class) is reported. The number inside the parenthesis in the legends of the images denote the length of the LOT feature vector corresponding to the particular choice of references. In all figures, for comparison, the results for classification using the semi discrete linear optimal transport framework \cite{merigot20} which uses the uniform measure as the reference is also reported. Standard deviations for each of the corresponding classification tests are reported in the Appendix Figure \ref{multirefImage79std}.}}
    \label{multirefImage79}
\end{figure}

\begin{figure}[H]
    \centering
    \includegraphics[width=0.5\textwidth]{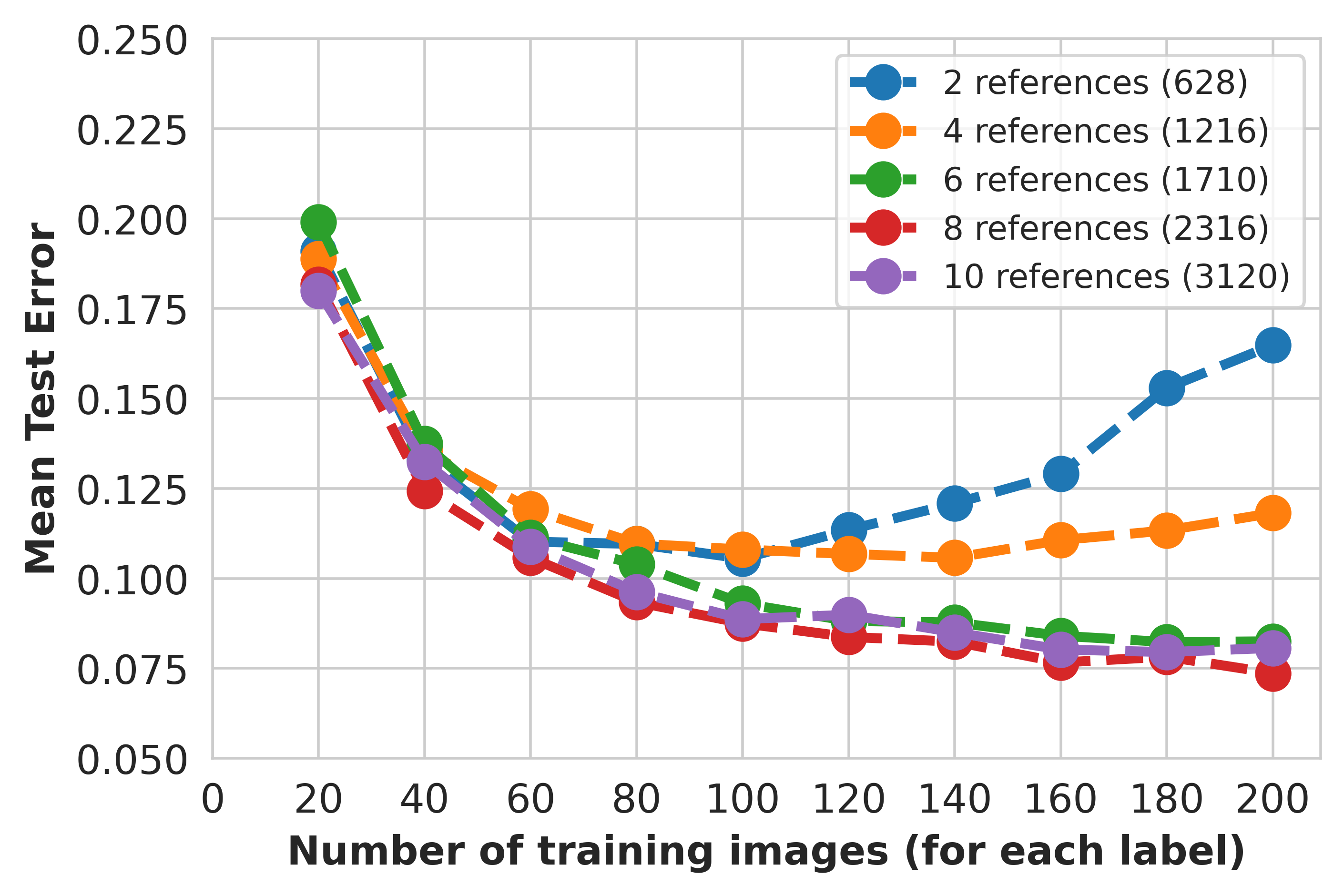}
    \caption{\scriptsize{Illustration of the benefit of using multiple references to reduce overfitting in the classification of severely sheared MNIST 7s and 9s using true MNIST images as references under the same training and testing conditions of Figure \ref{multirefImage79} (b3).}}
    \label{overfitting}
\end{figure}

\begin{figure}[H]
    \centering
    \includegraphics[width=\textwidth]{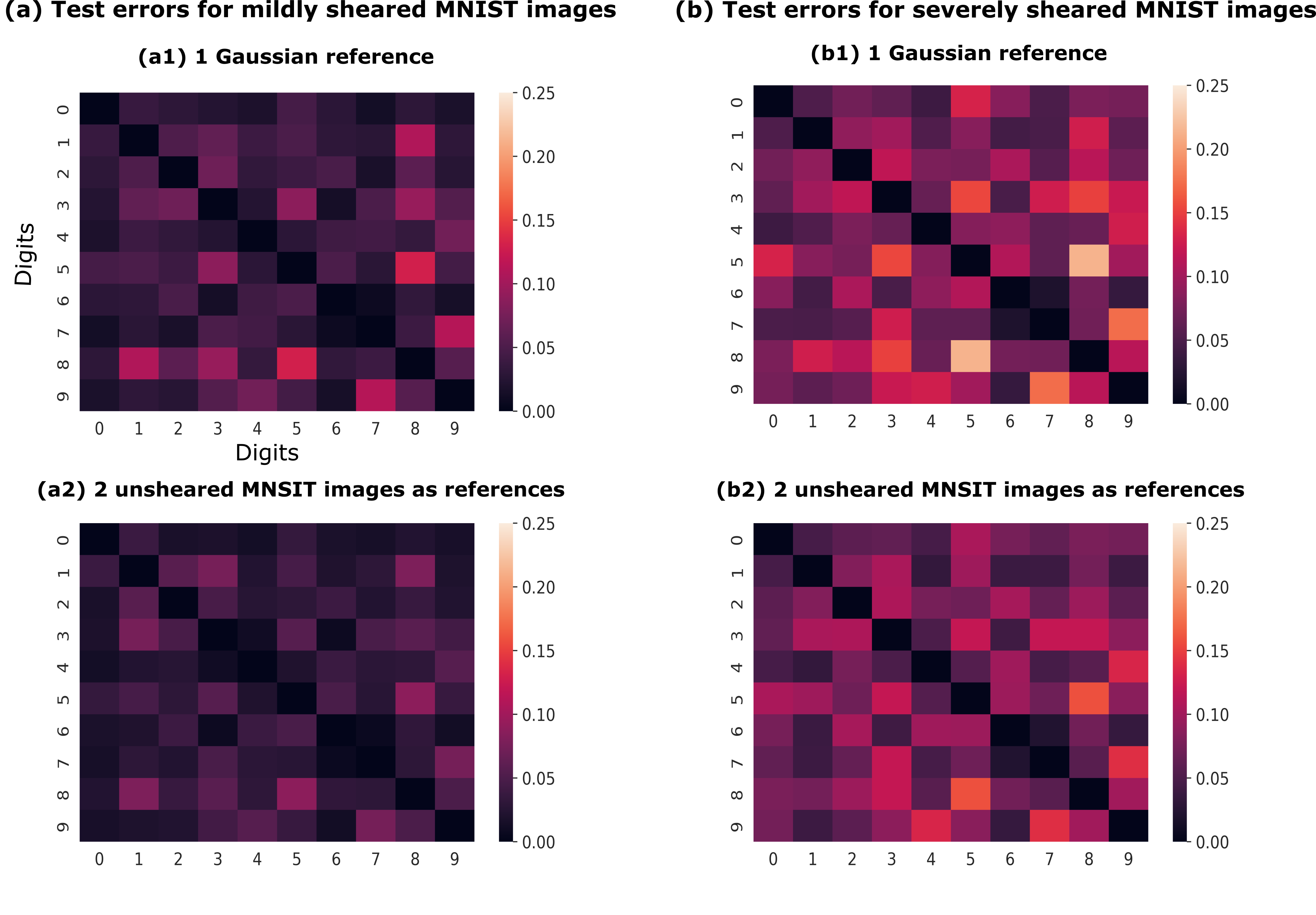}
    \caption{\scriptsize{(a) Test errors for binary classification of all pairs of mildly sheared MNIST images using (a1) one Gaussian reference (a2) two unsheared MNIST images as references. (b) Test errors for binary classification of all pairs of mildly sheared MNIST images using (b1) one Gaussian reference (b2) two unsheared MNIST images as references. For each given pair of digits, in the case of MNIST images as references (a2),(b2), one image corresponding to each class is randomly drawn to serve as the references. The reported error is a mean value 20 experiments involving different choices of 50 randomly drawn training images per class and 500 randomly drawn test images per class for each experiment. The range of standard deviations for the test errors for each case is reported in Table \ref{heatMapTable}.}}
    \label{heatmap}
\end{figure}

\begin{table}[h]
 \scriptsize
 \begin{tabular}{||c|c c |c c||} 
 \hline
 \multirow{2}{3cm}{\textbf{Reference choice}} & \multicolumn{2}{c|}{\textbf{Range of mean test errors}} & \multicolumn{2}{c||}{\textbf{Range of std.deviation in test errors}}  \\ 
 & Mild shearing & Severe shearing & Mild shearing & Severe shearing \\
 \hline
 1 Gaussian reference & $[0.0083,0.1298]$ & $[0.0198,0.2132]$ & $[0.0064,0.0291]$ & $[0.0108,0.0382]$\\
 2 unsheared MNIST references & $[0.0078,0.0880]$ & $[0.0220,0.1585]$ & $[0.0056,0.0244]$ & $[0.0111,0.0328]$\\
 \hline
 \end{tabular}
 \caption{\scriptsize{Range of mean value and standard deviations of test errors for pairwise classification of sheared MNIST images across all pairs of digits for various reference choices. The reported values are across 20 experiments involving different choices of 50 randomly drawn training images per class and 500 randomly drawn test images per class for each experiment.}}
 \label{heatMapTable}
\end{table}

\begin{figure}[h]
    \centering
    \includegraphics[width=\textwidth]{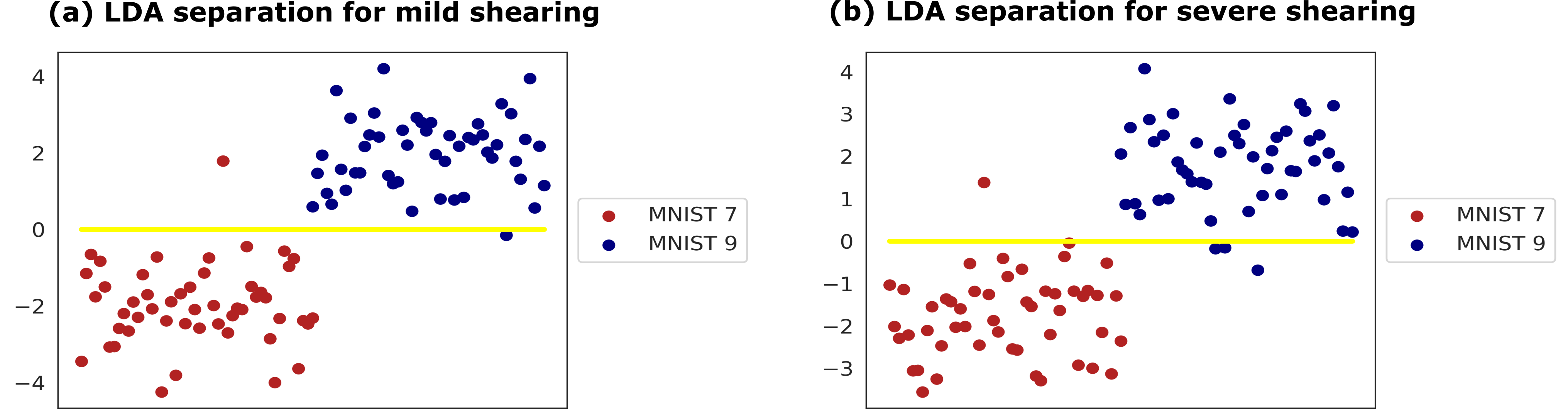}
    \caption{\footnotesize{Visualization of separation in the LDA projection using LOT features with $2$ unsheared MNIST images as references for 50 training images per class corresponding to (a) mildly sheared MNIST sevens and nines (b) severely sheared MNIST sevens and nines. The y-axis denotes the value of the projection onto the LDA separating line for the two classes.}}
    \label{severeLDA}
\end{figure}

%%%%%%%%%% ACKNOWLEDGEMENTS %%%%%%%%%%%%%%%%
\subsection*{Acknowledgements}
This research is supported by NSF awards DMS-1819222 and DMS-2012266, by Russell Sage Foundation Grant 2196 (to AC), and by NSF award DMS-2111322 (to CM).

%%%%%%%%%%%%% BIBLIOGRAPHY %%%%%%%%%%%%%%%

\bibliographystyle{abbrv}
%\bibliographystyle{plainnat}
%\bibliography{lit}

%%%%%%%%%%%%% APPENDIX A %%%%%%%%%%%%%%%
\begin{appendix}
\section{Compatibility Condition Proofs}

%%%%%%%% PLEASE ADD REFERENCE %%%%%%%%%%%%%%%%%%%%%%

\begin{lemma}\label{invariantLemma}
Suppose $V$ is a finite-dimensional vector space, $\phi: V \to V$ is a diagonalizable linear map, and $U \subseteq V$ is a $\phi$-invariant subspace.  Then the restriction $\phi\vert_U: U \to U$ is diagonalizable.
\end{lemma}

\begin{proof}
Let $\lambda_1, \dotsc, \lambda_k$ be distinct eigenvalues of $\phi$.  We will denote by $E(\lambda_k, \phi)$ the eigenspace of $\phi$ corresponding to eigenvalue $\lambda_k$.  Since $\phi$ is diagonalizable over $V$, we can represent $V$ as a direct sum
\begin{align*}
    V = E(\lambda_1, \phi) \oplus \dotsb \oplus E(\lambda_m, \phi).
\end{align*}
This means exactly that any vector $v$ is given by
\begin{align*}
    v = w_1 + \dotsb + w_m
\end{align*}
where $w_i \in E(\lambda_i, \phi)$.  As $U$ is a finite dimensional vector space, we know that there exists a basis for $U$ given by $\{u_1, \dotsc, u_k\}$.  Let us consider the linear map
\begin{align*}
    \Phi_i(u) = \prod\limits_{\substack{ j = 1,\dotsc, m \\ j \neq i} } (\lambda_j I - \phi\vert_U ) u.
\end{align*}
Note that this linear map is commutative in its order of composition.  We now will take every basis vector $u_i$ and represent it in terms of eigenvector.  Note that because $u_i$ is a vector in $V$, we find that there exists eigenvectors $w_{1,1} \in E(\lambda_1, \phi), \dotsc, w_{1,m} \in E(\lambda_m, \phi)$ such that
\begin{align*}
    u_1 = w_{1,1} + w_{1,2} + \dotsc, w_{1,m}.
\end{align*}
Now let us create a set $\widehat{W}_1 = \{w_{1,1}, \dotsc, w_{1,2}\}$.  Note that
\begin{align*}
    \Phi_i(u_1) = \prod\limits_{\substack{ j = 1,\dotsc, m \\ j \neq i} } (\lambda_j  - \lambda_i ) w_{1,i} \implies w_{1,i} \in U,
\end{align*}
since $U$ is $\phi$-invariant.  Because this happens for arbitrary $i$, we know that $w_{1,i} \in U$ for all $i$.  Note, that this set is linearly independent since each $w_{1,i}$ comes from a different eigenspace.  We repeat this for $u_j$ to obtain $\widehat{W}_j$, and note that $\widehat{W}_j \subseteq U$.  Now, let us define $\bigcup\limits_{j=1}^k \widehat{W}_j = \widehat{W}$.  Note that this is a spanning set of eigenvectors for $U$, and we can make this into a linearly independent set that still spans $U$ by throwing away the linearly dependent vectors.  Note that because of finite dimensionality, this process will stop, and will yield a linearly independent, spanning set of $U$, let's call it $\widehat{W}$, consisting of eigenvectors.  So this means that $\phi\vert_U$ is diagonalizable since we found an eigenbasis for $U$.  So we're done.
\end{proof}
The following theorem is a fundamental result from matrix analysis (see \cite[Theorem 1.3.12]{hornandjohnson}), but we provide a proof for convenience of the reader.
\begin{theorem}\label{prelimThm}
Let $A$ and $B$ be two $n \times n$ diagonalizable matrices that commute (i.e. $AB = BA$).  Then there exists a basis of $\RR^n$ consisting of simultaneous eigenvectors of $A$ and $B$.
\end{theorem}
\begin{proof}
We break this proof up into two parts.  First we will show that given an eigenvector $\lambda$ the eigenspace of $A$ corresponding to $\lambda$ (we denote this with $E(\lambda, A)$ is $B$-invariant.  Consider $v \in E(\lambda, A)$, then notice that
\begin{align*}
    A B v = B A v = B (\lambda v) = \lambda B v.
\end{align*}
This means that $Bv$ is an eigenvector for $A$ with eigenvalue $\lambda$, which means that $E(\lambda,A)$ is $B$-invariant since $B$ maps elements of $E(\lambda,A)$ back into $E(\lambda,A)$.  Now we show that there exists a basis for $\RR^n$ consisting of simultaneous eigenvectors of $A$ and $B$.

Note that because $A$ is diagonalizable, we know that $\RR^n$ can be represented as a direct sum given by
\begin{align*}
    \RR^n = \bigoplus\limits_{i=1}^k E(\lambda_i, A),
\end{align*}
where $\lambda_1, \dotsc, \lambda_k$ are distinct eigenvalues of $A$.  Now to show that there exists a basis of $\RR^n$ consisting of simultaneous eigenvectors of $A$ and $B$, we only need to find a basis for each subspace $E(\lambda, A)$ because the concatenation of all these bases will yield a basis for $\RR^n$.  Now note that since $E(\lambda,A)$ is a $B$-invariant space by above and because $B$ is diagonalizable, we know from \Cref{invariantLemma} that the restriction of $B$ to this eigenspace, $B\vert_{E(\lambda, \phi)}$, is diagonalizable, which means that there exists an eigenbasis of $E(\lambda,A)$ for the map $B$.  Let us call this this eigenbasis $S_{\lambda,A} = \{w_1, \dotsc, w_j\}$, where $j$ is the dimension of $E(\lambda, A)$.  Now, note that $S_{\lambda,A}$ consists of eigenvectors of both $B$ and $A$.  To see this, note that $S_{\lambda,A} \subseteq E(\lambda, A)$; thus, every $w_i$ is an eigenvector of $A$.  Moreover, $S_{\lambda, A}$ is an eigenbasis for $B\vert_{E(\lambda, A)}$ by construction (from \Cref{invariantLemma}).  This means that
\begin{align*}
    S = \bigcup\limits_{i=1}^k S_{\lambda_i, A}
\end{align*}
forms a basis for $\RR^n$ consisting of simultaneous eigenvectors of $A$ and $B$.
\end{proof}

\begin{lemma}\label{q=p}
If two symmetric matrices $A$ and $B$ commute, then there exists spectral decompositions $A = Q^\top \Lambda Q$ and $B = P^\top D P$ such that the rows of $Q$ are the same as the rows of $P$ up to a permutation.
\end{lemma}
\begin{proof}
We already know that if two diagonalizable matrices commute, then they share the same eigenvectors; thus, there exist an eigendecomposition for $A$ and $B$ with the same eigenvectors.  By extension, this holds for symmetric matrices.  If we assume that these eigendecompositions are given by $A = Q^\top \Lambda Q$ and $B = P^\top D P$, the eigenvectors of $A$ are exactly the columns of $Q^\top$, and similarly, the eigenvectors of $B$ are exactly the columns of $P^\top$.  This implies that the columns of $Q^\top$ and $P^\top$ should be the same.  The order of the columns can be permuted without loss of generality and still provide the same transformation $A$ and $B$.  Thus, we can assume that $Q$ has the same rows as $P$.
\end{proof}

\begin{theorem}\label{main}
Let $S: \RR^n \to \RR^n$ be a differentiable map such that $S = \nabla \varphi$ for some $\varphi$.  Let $\refe, \temp \in \mathcal{P}_2(\RR^n)$ with $\refe$ absolutely continuous with respect to the Lebesgue measure.  Assume that the compatibility condition $S \circ T_\refe^\temp = T_\refe^{S_\sharp \temp}$ holds.  Then $J_S(x)$ is a symmetric positive definite matrix for all $x$.  Moreover, $J_S(T_\refe^\temp(x))$, $J_{T_\refe^\temp}(x)$, and $J_{T_\refe^{S_\sharp \temp}}(x)$ share the same eigenspaces.  Furthermore, the eigenvalues of $J_S(T_\refe^\temp(x))$ are of the form $\frac{\lambda_{\refe,\temp}}{\lambda_{\refe, S_\sharp \temp}}$ where $\lambda_{\refe, \temp}$ is an eigenvalue of $J_{T_\refe^\temp}(x)$ and $\lambda_{\refe, S_\sharp \temp}$ is an eigenvalue of $J_{T_\refe^{S_\sharp \temp}}(x)$.
\end{theorem}

\begin{proof}[Proof of \cref{main}]\label{proofOfMain}  Recall that the main equation for us to study is
\begin{align*}
    S \circ T_\refe^\temp = T_\refe^{S_\sharp \temp}.
\end{align*}
By \Cref{Brenier}, there exist convex functions $\gamma$ and $\phi$ such that $T_\refe^\temp = \nabla \phi$ and $T_\refe^{S_\sharp \temp} = \nabla \gamma$.  By Clairaut's theorem (or the Schwarz theorem), $\nabla^2 \gamma(x)$ and $\nabla^2 \phi(x)$ are symmetric.  Using the multivariate chain rule and the symmetry of $\nabla^2 \gamma(x)$, we get that
\begin{align*}
    \nabla^2 \gamma(x) &= J_S( \nabla \phi(x)) \nabla^2 \phi(x) \\
    \nabla^2 \gamma(x)^\top &= (\nabla^2 \phi(x))^\top  J_S( \nabla \phi(x))^\top \\
    &= \nabla^2 \phi(x) J_S(\nabla \phi(x))^\top.
\end{align*}
Since $J_S = \nabla^2 \varphi$ for some $\varphi$, then $J_S^\top(x) = J_S(x)$ for all $x \in \RR^d$.  Since $J_S(\nabla \phi(x))$ and $\nabla^2 \phi(x)$ are symmetric  matrices that commute, according to \Cref{q=p}, there exists some orthogonal matrix $P$ such that we can write the eigendecompositions of $\nabla^2 \phi(x)$ and $J_S(\nabla \phi(x))$ as $\nabla^2 \phi(x) = P^\top \Lambda_\phi(x) P$ and $J_S(\nabla \phi(x)) = P^\top \Lambda_S(\nabla \phi(x)) P$ where the matrices $\Lambda_\phi$ and $\Lambda_S$ are diagonal matrices with the eigenvalues of $\nabla^2 \phi(x)$ and $J_S(\nabla \phi(x))$, respectively.  Moreover, if $\Lambda_\gamma$ denotes the diagonal matrix in the eigendecomposition for $\gamma$, then our matrix equations above can be written as
\begin{align*}
    \nabla^2 \gamma(x) &= J_S( \nabla \phi(x)) \nabla^2 \phi(x) \\
    P^\top \Lambda_\gamma(x) P &= P^\top \Lambda_S(\nabla \phi(x)) P P^\top  \Lambda_\phi(x) P \\
    \Lambda_\gamma(x) &= \Lambda_S(\nabla \phi(x)) \Lambda_\phi(x).
\end{align*}
This immediately shows that every eigenvalue $\lambda_S$ of $J_S(\nabla \phi(x))$ can be written as $\frac{\lambda_\gamma}{\lambda_\phi}$, where $\lambda_\gamma$ is an eigenvalue of $\nabla^2 \gamma(x)$ and $\lambda_\phi$ is an eigenvalue of $\nabla^2 \phi(x)$.  Since $\nabla^2 \phi(x)$ and $\nabla^2 \gamma(x)$ are Hessians of a convex function, they must be positive definite.  This implies that all the eigenvalues of $J_S(\nabla \phi(x))$ are positive.  Since $J_S(\nabla \phi(x))$ is symmetric, we immediately get that $J_S( \nabla \phi(x)) = \nabla^2 \varphi ( \nabla \phi(x) )$ is a symmetric positive definite matrix, which means that $\varphi$ must have been convex.  This implies that $S = \nabla \varphi$ is a transport map.
\end{proof}

\begin{lemma}\label{lemmaForShear}
Let an optimal transport map be given by $\nabla \phi(x)$ for some convex function $\phi$.  If the Hessian $\nabla^2 \phi(x)$ has a spectral decomposition that does not depend on $x$ (i.e. $P^\top D(x) P$ for a positive diagonal matrix $D(x)$), then the map $P \nabla \phi( P^\top x)$ has a diagonal Jacobian and each component of $P \nabla \phi( P^\top x)$ is a function of only a single variable.

%In this case, we can effectively assume that $P = I$ since we can apply appropriate change of bases to get to this form.
\end{lemma}

\begin{proof}[Proof of \cref{lemmaForShear}]
If we compute the Jacobian of $P \nabla \phi( P^\top x)$ by using the chain rule twice, we get that the Jacobian of $P \nabla \phi (P^\top x)$ is given by
\begin{align*}
    J_{P \nabla \phi (P^\top x)}( x) &= P J_{\nabla \phi (P^\top x)} ( x) = P \nabla^2 \phi( P^\top x) P^\top \\
    &= P P^\top D(P^\top x) P P^\top = D(P^\top x).
\end{align*}
This means that if we write the transport map $\nabla \phi$ in the basis given by the columns of $P^\top$ and the output is written in terms of the basis given by the columns of $P$, our transport map $\nabla \phi$ can be written as $n$ single variable functions.  To see this, notice that we can write the $j$th coordinate output of $P \nabla \phi(P^\top x)$ as some function $f_j$ to give us
\begin{align*}
    P \nabla \phi(P^\top x) = \begin{bmatrix}
    f_1(x_1,\dotsc, x_n) \\ f_2(x_1, \dotsc, x_n) \\ \vdots \\ f_n(x_1, \dotsc, x_n)
    \end{bmatrix}.
\end{align*}
Recall that the $(j,k)$th entry of the Jacobian $J_{P\nabla \phi( P^\top x)}(x)$ is $\frac{\partial f_j}{\partial x_k}$.  Because the Jacobian is diagonal, we see that $\frac{\partial f_j}{\partial x_k} = 0$ for $j \neq k$.  This implies that we can actually write
\begin{align*}
    P \nabla \phi(P^\top x) = \begin{bmatrix}
    f_1(x_1) \\ f_2(x_2) \\ \vdots \\ f_n(x_n)
    \end{bmatrix}.
\end{align*}
So we're done.
\end{proof}

\vspace{0.2cm}

Now we can prove the main LOT isometry theorems for shears.

\begin{proof}[Proof of \Cref{converseShearTheorem}]\label{ProofConverseShearTheorem}
Assume that the Jacobian of $T_\refe^\temp$ has constant orthonormal basis given by an orthogonal matrix $P$, then \Cref{main} tells us that a compatible transformation $S$ must have positive symmetric definite Jacobian $J_S$ and has the same eigenspaces as $J_{T_\refe^\temp}$.  First, note that the corollaries of \Cref{main} implies that $S$ is an optimal transport map.  Second, note that since $J_S$ commutes with $J_{T_\refe^\temp}$, we know that $J_S = \Tilde{P}^\top D(x) \Tilde{P}$, where $\Tilde{P}$ is a row-permutation of $P$ from \Cref{q=p}.  Because $S$ satisfies the assumptions of \cref{lemmaForShear}, we get that
\begin{align*}
    \Tilde{P} S( \Tilde{P}^\top x) &= \begin{bmatrix}
    f_1(x_1) \\ f_2(x_2) \\ \vdots \\ f_n(x_n)
    \end{bmatrix} \\
    \implies S(x) &= \Tilde{P}^\top \begin{bmatrix}
    f_1((\Tilde{P}x)_1) \\ f_2((\Tilde{P}x)_2) \\ \vdots \\ f_n((\Tilde{P}x)_n)
    \end{bmatrix}
\end{align*}
for $f_j$ increasing and differentiable.  Note that $f_j$ differentiable because $J_S$ is assumed to exist, and $f_j$ is increasing because $J_S$ is positive definite.  The form of $S$, however, is exactly the form of an element of $\mathcal{F}(P)$ in \Cref{def:shears} (the constant vector $b$ is a constant of integration).  This proves \Cref{converseShearTheorem}.
\end{proof}

\begin{proof}[Proof of \Cref{shearTheorem}]\label{ProofshearTheorem}
Let us assume that our elementary transformation is $S(x) = P^\top g(Px)$, then note that the Jacobian of $S$ can be given as $J_S(x) = P^\top J_g(P x) P$, where $J_g(z) = \diag((g_j'(z_j))_{j=1}^{n})$ (i.e. $J_g$ is a diagonal matrix).  Now given our template $\temp$, let's assume that there exists a reference $\refe$ such that the compatibility $S \circ T_\refe^\temp = T_\refe^{S_\sharp \temp}$ holds, then we will try to get some necessary conditions that $\refe$ must satisfy.  In particular, from Theorem \ref{Brenier} we can write $T_\refe^\temp = \nabla \phi$ for some convex $\phi$; moreover, we know that the Hessian can be written as $\nabla^2 \phi(x) = Q^\top(x) D(x) Q(x)$ for some orthogonal matrix-valued function $Q(x)$ and diagonal matrix-valued function $D(x)$.  Now, using \cref{main}, we know that if $S \circ T_\refe^\temp = T_\refe^{S_\sharp \temp}$, then
\begin{align*}
    J_S(\nabla \phi(x)) \nabla^2 \phi(x) &= \nabla^2 \phi(x) J_S(\nabla \phi(x)) \\
    P^\top J_g(Px) P Q(x)^\top D(x) Q(x) &= Q(x)^\top D(x) Q(x) P^\top J_g(Px) P.
\end{align*}
Since $J_S(\nabla \phi(x))$ and $\nabla^2 \phi(x)$ are two symmetric matrices that commute, we can assume without loss of generality that $Q(x)$ is a row-permutation of $P$ for all $x$ by invoking \Cref{q=p}.  We can call this matrix $\Tilde{P}$.  In particular, we can write $\nabla^2 \phi(x) = \Tilde{P}^\top D(x) \Tilde{P}$, where $D(x) = \diag( d(x))$ for a vector-valued function $d(x)$ with $d_i(x) > 0$ (the positivity comes from the fact that the Hessian must have positive eigenvalues).

We see that since $\nabla^2 \phi(x)$ has a constant eigendecomposition, we know from \Cref{lemmaForShear} that 
\begin{align*}
    \Tilde{P} \nabla \phi(\Tilde{P}^\top x) &= \begin{bmatrix}
    f_1(x_1) \\ f_2(x_2) \\ \vdots \\ f_n(x_n)
    \end{bmatrix} \\
    \implies \nabla \phi(\Tilde{P}^\top x) &= \Tilde{P}^\top \begin{bmatrix}
    f_1(x_1) \\ f_2(x_2) \\ \vdots \\ f_n(x_n)
    \end{bmatrix}.
\end{align*}
From \cref{lemmaForShear}, we also note that a choice of the diagonals $d_j(x_j) > 0$ gives a unique (up to a constant) anti-derivative $f_j = \int d_j(x_j) dx_j$.  Thus, without loss of generality, we can consider $f_j$'s to be completely determined by the $d_j$'s.

If we assumed that our inputs $x$ are actually written in the basis given by $\Tilde{P}^\top$ and the outputs are written in basis given by $\Tilde{P}$, then our map transport map decomposes into $n$ single-variable functions as shown above.  Moreover, note that $f_j(x_j)$ must be an increasing function since $\frac{\partial f_j}{\partial x_j} > 0$ everywhere.  Thus, in principle, this map must be invertible, and we can actually compute the inverse of this map by computing
\begin{align*}
     y = \begin{bmatrix}
    y_1 \\ y_2 \\ \vdots \\ y_n \end{bmatrix} = \nabla \phi(x) &= \nabla \phi(\Tilde{P}^\top \Tilde{P} x) = \Tilde{P}^\top \begin{bmatrix}
    f_1((\Tilde{P}x)_1) \\ f_2((\Tilde{P}x)_2) \\ \vdots \\ f_n((\Tilde{P}x)_n)
    \end{bmatrix} \\
    \Tilde{P}y &= \begin{bmatrix}
    f_1((\Tilde{P}x)_1) \\ f_2((\Tilde{P}x)_2) \\ \vdots \\ f_n((\Tilde{P}x)_n)
    \end{bmatrix} \\
    \begin{bmatrix}
    f_1^{-1} ( (\Tilde{P}y)_1 ) \\ f_2^{-1} ( (\Tilde{P}y)_2 ) \\ \vdots \\ f_n^{-1}( (\Tilde{P}y)_n )
    \end{bmatrix} &= \begin{bmatrix}
    (\Tilde{P}x)_1 \\ (\Tilde{P}x)_2 \\ \vdots \\ (\Tilde{P}x)_n
    \end{bmatrix} = \Tilde{P}x \\
    \implies \nabla \phi^{-1}( y ) &= \Tilde{P}^\top \begin{bmatrix}
    f_1^{-1} ( (\Tilde{P}y)_1 ) \\ f_2^{-1} ( (\Tilde{P}y)_2 ) \\ \vdots \\ f_n^{-1}( (\Tilde{P}y)_n )
    \end{bmatrix}.
\end{align*}
Note that because the inverse of an increasing function is also increase, we have that $\nabla \phi^{-1} \in \mathcal{F}(P)$.  In practice, we will be given $S$ and $\mu$; thus, we would want to find $\refe$ such that $T_\sigma^\mu$ is compatible with $S$.  Note that this will be exactly given by the map $\nabla \phi^{-1}(y)$ because $\refe = \nabla \phi^{-1}_\sharp \mu$.  This proves \Cref{shearTheorem}.
\end{proof}

\begin{proof}[Proof of \cref{orthogonalImpliesIdentity}]\label{ProoforthogonalityImpliesIdentity}
Given our elementary transformation $S(x) = Ax + b$, we have that $J_S = A$.  Theorem \ref{main}, however, shows us that $A$ must be positive symmetric definite.  We will show that the only matrix $A$ that is both positive symmetric definite and orthogonal is the identity.  To see this note that since $A$ is symmetric, we know that $A^\top = A$.  Since $A$ is assumed to be orthogonal, we know that $A^\top A = A^2 = I$.  Let $v$ be an eigenvector of $A$ with eigenvalue $\lambda$, then $v = A^2 v = \lambda^2 v$.  This means that $\lambda^2 = 1$.  Since $A$ is symmetric, we know that all the eigenvalues must be real; thus, $\lambda = \pm 1$.  Moreover, because $A$ is positive symmetric definite, the only eigenvalue it could be are $+1$.  This implies that $A$ is the identity.  In particular, this means that constant rotations are not valid elementary transformations for which the compatibility condition holds.
\end{proof}

%%%%%%%%%%%%% APPENDIX B %%%%%%%%%%%%%%%

\section{Proofs of Separability Results}

For a set of measures $\temp_1$ and $\temp_2$ and a set of elementary transformations $\mathcal{H}$, the general method of showing that $F_{\sigma}(\mathcal{H} \star \temp_1)$ and $F_{\sigma}(\mathcal{H} \star \temp_2)$ are linearly separable is to
\begin{enumerate}
    \item Show that $\mathcal{H}$ is convex,
    \item Show that $\mathcal{H} \star \temp_1$ and $\mathcal{H} \star \temp_2$ are compact (or at least have their closures as being compact),
    \item Show that $W_2(\mathcal{H} \star \temp_1, \mathcal{H} \star \temp_2) > \delta$ for some $\delta > 0$.
\end{enumerate}
 We show this now for shears, but for another class of elementary transformations, we must show that $\mathcal{H}$ is convex.

\begin{lemma}
The set of shears $\mathcal{H}_{\gamma, M, M_b}$ described in \cref{shearsH} is convex.
\end{lemma}
\noindent\textbf{Proof:}  Let $h, h' \in \mathcal{H}_{\gamma, M, M_b}$ and $s \in [0,1]$, then we want to show that $s h + (1-s) h' \in \mathcal{H}_{\gamma, M, M_b}$.  We find that
\begin{align*}
    s h(x) + (1-s)h'(x) &= s(A x + b) + (1-s)( A' x + b') \\
    &= (sA +(1-s)A')x + (sb + (1-s)b').
\end{align*}
Notice first that $sA + (1-s)A'$ is symmetric.  Moreover, note that
\begin{align*}
    \lambda_{\min}(sA + (1-s)A') &= \min_{\Vert x \Vert_2 = 1} \< x, (s A + (1-s)A') x \> \\
    &= \min_{\Vert x \Vert_2 = 1} s \<x, Ax\> + (1-s) \<x, A' x\> \\
    &\geq s \underbrace{\min_{\Vert x \Vert_2 = 1} \<x, Ax\>}_{\geq \lambda_{\min}(A) } + (1-s) \underbrace{\min_{\Vert \Tilde{x} \Vert_2 = 1} \<\Tilde{x}, A' \Tilde{x}\>}_{\geq \lambda_{\min}(A')} \\
    &\geq s \lambda_{\min}(A) + (1-s)\lambda_{\min}(A') > s\gamma + (1-s)\gamma = \gamma;
\end{align*}
and similarly,
\begin{align*}
    \lambda_{\max}(sA + (1-s)A') &= \max_{\Vert x \Vert_2 = 1} \< x, (s A + (1-s)A') x \> \\
    &= \max_{\Vert x \Vert_2 = 1} s \<x, Ax\> + (1-s) \<x, A' x\> \\
    &\leq s \underbrace{\max_{\Vert x \Vert_2 = 1} \<x, Ax\>}_{\leq \lambda_{\max}(A) } + (1-s) \underbrace{\max_{\Vert \Tilde{x} \Vert_2 = 1} \<\Tilde{x}, A' \Tilde{x}\>}_{\leq \lambda_{\max}(A')} \\
    &\leq s \lambda_{\max}(A) + (1-s)\lambda_{\max}(A') < sM + (1-s)M = M.
\end{align*}
This means that $sA + (1-s) A'$ is symmetric positive definite and actually has the correct bounds on its eigenvalues.  We now show that $sb + (1-s)b'$ satisfies the proper bounds too.  Notice that
\begin{align*}
    \Vert s b + (1-s)b' \Vert_2 \leq s \Vert b \Vert_2 + (1-s) \Vert b' \Vert_2 \leq s M_b + (1-s) M_b = M_b.
\end{align*}
This implies that $sh + (1-s) h' \in \mathcal{H}$. So we're done. \qed 

Next, given a base measure $\temp$ and set of elementary transformations $\mathcal{H}$, we ideally want to show that the set $\mathcal{H} \star \temp = \{h_\sharp \temp : h \in \mathcal{H}\}$ is compact, but the weaker condition of $\mathcal{H} \star \temp$ being precompact should be good enough for our purposes.  To address compactness, we need a definition.

\begin{definition}[Tightness]
Let $(X, \mathcal{T})$ be a Hausdorff space and let $\mathcal{S}$ be a $\sigma$-algebra such that $\mathcal{T} \subseteq \mathcal{S}$.  Let $M$ be a collection of probability measures defined on $\mathcal{S}$.  The collection $M$ is called \textbf{tight} if, for any $\epsilon > 0$, there exists a compact subset $K_\epsilon \subset X$ such that for all measures $\mu \in M$, we have $\mu(K_\epsilon) > 1 - \epsilon$.
\end{definition}

A natural theorem that relates tightness of measures to compactness is Prokhorov's theorem.

\begin{theorem}[Prokhorov]
Let $(X, d)$ be a a separable metric space.  Let $\mathcal{P}(X)$ be the collection of all probability measures defined on $X$ with respect to the Borel $\sigma$-algebra.  Then a collection $\mathcal{K} \subset \mathcal{P}(X)$ of probability measures is tight if and only if the closure of $\mathcal{K}$ is sequentially compact in $\mathcal{P}_2(X)$ equipped with the topology of weak convergence.
\end{theorem}

According to \cite[pp.\ 37--42]{panaretos-2020}, we can upgrade Prokhorov's theorem to be sequentially compact with the Wasserstein $2$-metric if
\begin{align*}
    \sup_{\mu \in \mathcal{K}} \int_{x: \Vert x \Vert_2 > R} \Vert x \Vert_2^2 d\mu(x) \xrightarrow{R \to \infty} 0.
\end{align*}
This is easily true if $\sup_{\mu \in \mathcal{K}} \{\Vert x \Vert_2 : x \in \text{supp}(\mu) \} \leq R < \infty$.

\begin{corollary}
Let $\mathcal{H}$ be a set of transformations such that for every $R > 0$, there exists $\Tilde{R}$ such that $\sup_{h \in \mathcal{H}} \{ \Vert h(x) \Vert_2 : \Vert x \Vert_2 < R \} < \Tilde{R}$.  Also assume that $\temp \in \mathcal{P}_2(\RR^n)$ has bounded support $R_\mu$, then $\mathcal{H} \star \temp$ is a precompact set of measures. 
\end{corollary}
\begin{proof}
For us, if $\temp$ has bounded support with bound $R_\mu$, we should have that all measures belonging to $\mathcal{H} \star \temp$ must also have support bounded for some $\Tilde{R}>0$.  To see this, note that for $\Tilde{\mu} \in \mathcal{H} \star \temp$, we have $\text{supp}(\Tilde{\temp})$ is bounded by $\Tilde{R}$ for some $\Tilde{R} > 0$.  So we're done.
\end{proof}

For shears, we can see that every measure from $\mathcal{H}_{\gamma, M, M_b} \star \temp_1$ and $\mathcal{H}_{\gamma, M, M_b} \star \temp_2$ has bounded support since $\sup_{h \in \mathcal{H}_{\gamma, M, M_b}} \{ \Vert h(x) \Vert : x \in \text{supp}(\temp), \temp \in \mathcal{K} \} \leq M R+M_b $.  It's easy to see that $\mathcal{H}_{\gamma, M, M_b} \star \temp$ is tight for a big enough ball $B_R(0) = \{x : \Vert x \Vert_2 \leq R\}$ if $\refe$ has bounded support.  This means that $\mathcal{H}_{\gamma, M, M_b} \star \temp$ is precompact with the Wasserstein 2-metric for any $\temp$ with bounded support.

By \cite{villani-2009} Corollary 5.23, the stability of optimal transport maps implies that $F_\refe$ is continuous; thus, we find that $F_\refe(\mathcal{H} \star \mu)$ is precompact if $\mathcal{H}\star \mu$ is precompact.  Note also that \cref{Brenier} above gives us a corollary.

\begin{corollary}\label{otMap_Rep}
Let $h:\RR^n \to \RR^n$ be a transformations that can be represented as the gradient of a convex function, then for $\refe$, an absolutely continuous measure with respect to the Lebesgue measure, we get that $h_\sharp \refe = T_\refe^{h_\sharp \refe}$.
\end{corollary}

Now we must show that $F_\refe(\mathcal{H} \star \mu)$ is convex, which will ensure that our LOT embedding is convex and precompact.
\begin{lemma}
Let $\mathcal{H} = \{ h: \RR^n \to \RR^n \vert h =\nabla \phi, \phi \text{ is convex}\}$ be a convex set of transformations and let $\refe$ and $\mu$ be absolutely continuous (with respect to the Lebesgue measure) probability measures, then $F_\refe(\mathcal{H} \star \mu)$ is convex.
\end{lemma}
\begin{proof}
Let $h, \hat{h} \in \mathcal{H}$ and $s \in [0,1]$ so that $T_\refe^{h_\sharp \mu}, T_\refe^{\hat{h}_\sharp \mu} \in F_\sigma(\mathcal{H} \star \mu)$.  Then we want to show that $s T_\refe^{h_\sharp \mu} + (1-s) T_\refe^{\hat{h}_\sharp \mu} \in F_\sigma(\mathcal{H} \star \mu)$.  First notice that by Brenier's theorem, there exists convex functions $\phi$ and $\hat{\phi}$ such that $\nabla \phi = T_\refe^{h_\sharp \mu}$ and $\nabla \hat{\phi} = T_\refe^{\hat{h}_\sharp \mu}$.  Note now that
\begin{align*}
    s \nabla \phi + (1-s)\nabla \hat{\phi} = \nabla (s \phi + (1-s) \hat{\phi})
\end{align*}
so that $s T_\refe^{h_\sharp \mu} + (1-s) T_\refe^{\hat{h}_\sharp \mu}$ is actually the gradient of a convex function.  Moreover, by the uniqueness of optimal transport maps as gradients of convex functions, we know that $s T_\refe^{h_\sharp \mu} + (1-s) T_\refe^{\hat{h}_\sharp \mu}$ is the unique optimal transport map that transports $\refe$ to its target distribution.  If this target distribution is of the form $\Tilde{h}_\sharp \mu$ for some $\Tilde{h} \in \mathcal{H}$, then our proof is done.  Indeed:
\begin{align*}
    \Big(s T_\refe^{h_\sharp \mu} + (1-s) T_\refe^{\hat{h}_\sharp \mu}\Big) \refe &=  s T_\refe^{h_\sharp \mu}\refe + (1-s) T_\refe^{\hat{h}_\sharp \mu}\refe \\
    &= s h_\sharp \mu + (1-s) \hat{h}_\sharp \mu \\
    &= \Big( s h + (1-s) \hat{h} \Big)_\sharp \mu.
\end{align*}
Since $sh + (1-s)\hat{h} \in \mathcal{H}$, we know that $s T_\refe^{h_\sharp \mu} + (1-s) T_\refe^{\hat{h}_\sharp \mu}$ is the unique optimal transport map that transports $\refe$ to $(s h + (1-s)\hat{h})_\sharp \mu$.  This means that
\begin{align*}
    s T_\refe^{h_\sharp \mu} + (1-s) T_\refe^{\hat{h}_\sharp \mu} \in F_\sigma(\mathcal{H} \star \mu).
\end{align*}
Thus, $F_\sigma(\mathcal{H} \star \mu)$ is convex.
\end{proof}

Using the lemma above, we get that $F_{\refe}(\mathcal{H} \star \temp_1)$ and $F_{\refe}(\mathcal{H} \star \temp_2)$ are both convex and have compact closures.  For our linear separability result, we now only need to make sure that $\inf_{h, h' \in \mathcal{H}} \Vert T_{\refe}^{h_\sharp \temp_1} - T_{\refe}^{h'_\sharp \temp_2} \Vert_{\refe} \geq \delta$ for some $\delta > 0$.  Ideally, given $W_2(\temp_1, \temp_2)$ and the level of separation $\delta > 0$ we want, we should be able to find bounds on the function class $\mathcal{H}$ that we are considering.  This leads us to \cref{separability}:

\begin{proof}[Proof of \cref{separability}]\label{separabilityProof}
Assume that we have $\Tilde{h}, \Tilde{h}^\star \in \mathcal{H}_\epsilon$, then
using the triangle inequality, we have
\begin{align*}
    W_2(\Tilde{h}_\sharp \temp_1, \Tilde{h}^\star_\sharp \temp_2) &\geq \vert W_2(\temp_1, \Tilde{h}^\star_\sharp \temp_2) - W_2(\Tilde{h}_\sharp \temp_1, \temp_1) \vert \\
    &\geq \bigg\vert \big\vert W_2(\temp_1, \temp_2) - W_2(\temp_2, \Tilde{h}^\star_\sharp \temp_2) \big\vert - W_2(\Tilde{h}_\sharp \temp_1, \temp_1) \bigg\vert,
\end{align*}
provided that the quantity in the left-hand side is positive.  Now, we know from \cite{moosmueller20} that $W_2(\mu, \nu) \leq \Vert F_\sigma(\mu) - F_\sigma(\nu) \Vert_\sigma$; thus, we have that
\begin{align*}
   \bigg\vert \big\vert W_2(\temp_1, \temp_2) - W_2(\temp_2, \Tilde{h}^\star_\sharp \temp_2) \big\vert - W_2(\Tilde{h}_\sharp \temp_1, \temp_1) \bigg\vert &\leq \Vert F_{\sigma}(\Tilde{h}_\sharp \temp_1) - F_{\sigma}(\Tilde{h}^\star_\sharp \temp_2) \Vert_{\sigma}.
\end{align*}
So if we lower bound the left-hand side by $\delta > 0$, then $\Vert F_{\sigma}(\Tilde{h}_\sharp \temp_1) - F_{\sigma}(\Tilde{h}^\star_\sharp \temp_2) \Vert_{\refe} \geq \delta > 0$.  This would imply that $F_{\sigma}(\mathcal{H}_\epsilon \star \temp_1)$ and $F_{\sigma}(\mathcal{H}_\epsilon \star \temp_2)$ is linearly separable by the Hahn-Banach theorem.

To get this bound, let us find a generic bound for $W_2(\Tilde{h}_\sharp \temp, \temp)$ when $\Tilde{h} \in \mathcal{H}_\epsilon$.  In particular, there exists $h \in \mathcal{H}$ such that $\Vert h - \Tilde{h} \Vert_\temp$; thus, we get
\begin{align*}
    W_2(\Tilde{h}_\sharp \temp, \temp) \leq W_2(\Tilde{h}_\sharp \temp, h_\sharp \temp) + W_2(h_\sharp \temp, \temp).
\end{align*}
First, since $h$ is the gradient of convex function and \cref{otMap_Rep}, we know that $T_\temp^{h_\sharp \temp} = h$.  This means that the compatibility condition holds, which further implies that
\begin{align*}
    W_2(\Tilde{h}_\sharp \temp, \temp) &= \Vert \Tilde{h} - I \Vert_\temp \leq L.
\end{align*}
Moreover, equation 2.1 of \cite{ambrosio2013user} says that
\begin{align*}
    W_2(\Tilde{h}_\sharp \temp, h_\sharp \temp) \leq \Vert h - \Tilde{h} \Vert_\temp < \epsilon.
\end{align*}
Because of our bounds, our results implies that
\begin{align*}
    L \leq \frac{W_2(\temp_1,\temp_2) - \delta }{2} - \epsilon &\leq W_2(\temp_1,\temp_2) - \delta - \epsilon \\
    \implies W_2(\temp_1, \temp_2) - W_2( \temp_2,  \Tilde{h}^\star_\sharp \temp_2) &\geq W_2(\temp_1, \temp_2) - W_2(\Tilde{h}^\star_\sharp \temp_2, h^\star_\sharp \temp_2) - W_2(h^\star_\sharp \temp_2, \temp_2) \\
    &\geq W_2(\temp_1,\temp_2) - L - \epsilon > \delta > 0 .
\end{align*}
Essentially, we were able to remove the absolute values because the quantity in the absolute value was positive.  This positivity of the absolute value implies that we can replace
\begin{align*}
    \bigg\vert \big\vert W_2(\temp_1, \temp_2) - W_2(\temp_2, \Tilde{h}^\star_\sharp \temp_2) \big\vert - W_2(\Tilde{h}_\sharp \temp_1, \temp_1) \bigg\vert
\end{align*}
with
\begin{align*}
    \bigg\vert W_2(\temp_1, \temp_2) - W_2(\temp_2, \Tilde{h}^\star_\sharp \temp_2) - W_2(\Tilde{h}_\sharp \temp_1, \temp_1) \bigg\vert.
\end{align*}
But note that
\begin{align*}
    W_2(\temp_1, \temp_2) - W_2(\temp_2, \Tilde{h}^\star_\sharp \temp_2) - W_2(\Tilde{h}_\sharp \temp_1, \temp_1) &\geq W_2(\temp_1, \temp_2) - 2L - 2\epsilon \geq \delta \\
    \iff L &\leq \frac{W_2(\temp_1, \temp_2) - \delta}{2} - \epsilon.
\end{align*}
This implies that
\begin{align*}
    \bigg\vert W_2(\temp_1, \temp_2) - W_2(\temp_2, \Tilde{h}^\star_\sharp \temp_2) - W_2(\Tilde{h}_\sharp \temp_1, \temp_1) \bigg\vert \geq \delta > 0.
\end{align*}
So we see that if $L \leq \frac{W_2(\temp_1, \temp_2) - \delta}{2} - \epsilon$, then we must have that $\Vert F_{\refe}(h_\sharp \temp_1) - F_{\refe}(h'_\sharp \temp_2) \Vert_{\refe} \geq \delta$.
\end{proof}

\begin{proof}[Proof of \cref{compat_sep}]\label{compat_sep_proof}
For the first statement, the linear separability result is immediate because the compatibility criteria ensures that the LOT distance and Wasserstein-2 distance are the same.  To see this, we note that $h \circ T_{\temp_1}^{\temp_2}$ is compatible with respect to the optimal transport between $\temp_1$ and $(h \circ T_{\temp_1}^{\temp_2})_\sharp \temp_1$ because
\begin{align*}
    T_{\temp_1}^{(h\circ T_{\temp_1}^{\temp_2})_\sharp \temp_1} = T_{\temp_1}^{h_\sharp \temp_2} = h \circ T_{\temp_1}^{\temp_2} = h \circ T_{\temp_1}^{\temp_2} \circ T_{\temp_1}^{\temp_1}.
\end{align*}
This means that from \cite{moosmueller20}, for $h, \Tilde{h} \in \mathcal{H}$ we have that
\begin{align*}
    \Vert T_{\temp_1}^{ \Tilde{h}_\sharp \temp_1} - T_{\temp_1}^{h_\sharp \temp_2} \Vert_{\temp_1} &= \Vert T_{\temp_1}^{\Tilde{h}_\sharp \temp_1} - T_{\temp_1}^{(h\circ T_{\temp_1}^{\temp_2})_\sharp \temp_1} \Vert_{\temp_1} \\
    &= W_2( \Tilde{h}_\sharp \temp_1, (h\circ T_{\temp_1}^{\temp_2})_\sharp \temp_1) \\
    &= W_2( \Tilde{h}_\sharp \temp_1, h_\sharp \temp_2).
\end{align*}
This proves the first statement.

For the second statement, let $h_\epsilon, \Tilde{h}_\epsilon \in \mathcal{H}_\epsilon$ such that $\Vert h - h_\epsilon \Vert_{\temp_1} < \epsilon$ and $\Vert \Tilde{h} - \Tilde{h}_\epsilon \Vert_{\temp_1}$ for $h, \Tilde{h} \in \mathcal{H}$.  We know that $\Vert F_{\temp_1}( (\Tilde{h}_\epsilon)_\sharp \temp_1) - F_{\temp_2} ( (h_\epsilon)_{\sharp} \temp_2) \Vert_{\temp_1} \geq W_2((\Tilde{h}_\epsilon)_\sharp \temp_1, (h_\epsilon)_{\sharp} \temp_2)$.  Now we know that
\begin{align*}
    W_2((\Tilde{h}_\epsilon)_\sharp \temp_1, (h_\epsilon)_{\sharp} \temp_2) \geq \bigg\vert \Big \vert W_2(\Tilde{h}_\sharp \temp_1, h_\sharp \temp_2) - W_2( \Tilde{h}_\sharp \temp_1, (\Tilde{h}_\epsilon)_\sharp \temp_1) \Big\vert - W_2(h_\sharp \temp_2, (h_\epsilon)_\sharp \temp_2) \bigg\vert.
\end{align*}
From equation 2.1 of \cite{ambrosio-2013}, we have that
\begin{align*}
    W_2(\Tilde{h}_\sharp \temp_1, (\Tilde{h}_\epsilon)_\sharp \temp_1) \leq \Vert \Tilde{h} - \Tilde{h}_\epsilon \Vert_{\temp_1} < \epsilon \\
    W_2( h_\sharp \temp_2, (h_\epsilon)_\sharp \temp_2) \leq \Vert h - h_\epsilon \Vert_{\temp_2} < \epsilon .
\end{align*}
Note that $W_2(\Tilde{h}_\sharp \temp_1, h_\sharp \temp_2) \geq \inf_{h,\Tilde{h} \in \mathcal{H}} W_2(\Tilde{h}_\sharp \temp_1, h_\sharp \temp_2) > 2\epsilon$.  This means that
\begin{align*}
    W_2((\Tilde{h}_\epsilon)_\sharp \temp_1, (h_\epsilon)_{\sharp} \temp_2) &\geq \bigg\vert \underbrace{\Big\vert  \underbrace{\inf_{h,\Tilde{h} \in \mathcal{H}} W_2(\Tilde{h}_\sharp \temp_1, h_\sharp \temp_2) - \epsilon}_{> 0} \Big\vert - \epsilon}_{>0} \bigg\vert > 0.
\end{align*}
So we have that $\Vert F_{\temp_1}( (\Tilde{h}_\epsilon)_\sharp \temp_1) - F_{\temp_2} ( (h_\epsilon)_{\sharp} \temp_2) \Vert_{\temp_1} > 0$.

For the third statement, we extend the lower bounds from above.  Because $h, \Tilde{h} \in \mathcal{H}$ are compatible, we have that $W_2( \Tilde{h}_\sharp \temp_1, h_\sharp \temp_2) = \Vert T_{\temp_1}^{\Tilde{h}_\sharp \temp_1} - T_{\temp_1}^{h_\sharp \temp_2} \Vert_{\temp_1}$.
Using the triangle inequality, we get
\begin{align*}
    \Vert T_{\temp_1}^{\Tilde{h}_\sharp \temp_1} - T_{\temp_1}^{h_\sharp \temp_2} \Vert_{\temp_1} &= \Vert \Tilde{h} - T_{\temp_1}^{h_\sharp\temp_2} \Vert_{\temp_1} \\
    &= \Vert \Tilde{h} - h - (T_{\temp_1}^{h_\sharp \temp_2} - h) \Vert_{\temp_1} \\
    &\geq \Big\vert \Vert T_{\temp_1}^{ h_\sharp \temp_2} - T_{\temp_1}^{h_\sharp \temp_1} \Vert_{\temp_1} - \Vert \Tilde{h} - h \Vert_{\temp_1}  \Big\vert.
\end{align*}
Because $h \in \mathcal{H}$ is chosen to be compatible with respect to $\temp_1$ and $\temp_2$, note that
\begin{align*}
    W_2(h_\sharp \temp_1, h_\sharp \temp_2) &= \Vert T_{\temp_1}^{h_\sharp \temp_1} - T_{\temp_1}^{h_\sharp \temp_2} \Vert_{\temp_1} = \Vert h - h \circ T_{\temp_1}^{\temp_2} \Vert_{\temp_1} \\
    &= \Vert h \circ (I - T_{\temp_1}^{\temp_2}) \Vert_{\temp_1} = \Vert h \Vert_{\vert \temp_1 - \temp_2 \vert} \\
    &= \bigg( \int \Vert h(x) \Vert_2^2 d \vert \temp_1 - \temp_2 \vert(x) \bigg)^{1/2} \\
    &\geq \sqrt{2} \bigg( \int \Vert x_0 - x\Vert_2^2 d \vert \temp_1 - \temp_2 \vert(x) \bigg)^{1/2}
\end{align*}
where the last equality came from a change of variable and the inequality comes from our assumption that $\Vert h(x) \Vert_2 \geq \sqrt{2} \Vert x - x_0 \Vert_2$.  Now we refer to Theorem 6.15 of \cite{villani-2009}, which says that for any $x_0 \in \RR^n$, we have
\begin{align*}
    W_2(\temp_1,\temp_2) \leq \sqrt{2} \bigg( \int \Vert x_0 - x\Vert_2^2 d \vert \temp_1 - \temp_2 \vert(x) \bigg)^{1/2} = f(x_0).
\end{align*}
We want to minimize the right hand side; thus, taking the derivative $\frac{d}{dx_0} f(x_0) = 0$, this reduces to
\begin{align*}
    0 &= 2 \int (x_0 - x) d\vert \temp_1 - \temp_2\vert(x) \\
    \implies x_0 &= \frac{1}{\vert \temp_1 - \temp_2 \vert(\RR^n)} \int x d\vert \temp_1 - \temp_2\vert(x) .
\end{align*}
Essentially $x_0$ is the mean of the measure $\vert \temp_1 - \temp_2 \vert$ after normalization.  So we have that $W_2(\temp_1, \temp_2) \leq W_2(h_\sharp \temp_1, h_\sharp \temp_2)$.
Since $W_2(\temp_1, \temp_2) - \sup_{h,\Tilde{h} \in \mathcal{H}} \Vert \Tilde{h} - h \Vert_{\temp_1} \geq \delta + 2\epsilon > \delta + \epsilon$, these computations imply that
\begin{align*}
    \Big \vert W_2(\Tilde{h}_\sharp \temp_1, h_\sharp \temp_2) - W_2( \Tilde{h}_\sharp \temp_1, (\Tilde{h}_\epsilon)_\sharp \temp_1) \Big\vert - W_2(h_\sharp \temp_2, (h_\epsilon)_\sharp \temp_2) \geq W_2(\Tilde{h}_\sharp \temp_1, h_\sharp \temp_2) - 2 \epsilon.
\end{align*}
is greater than
\begin{align*}
    W_2(\temp_1, \temp_2) - \sup_{h,\Tilde{h} \in \mathcal{H}} \Vert \Tilde{h} - h \Vert_{\temp_1} - 2\epsilon > 0.
\end{align*}
This implies that
\begin{align*}
    W_2((\Tilde{h}_\epsilon)_\sharp \temp_1, (h_\epsilon)_{\sharp} \temp_2) \geq \Big\vert W_2(\temp_1, \temp_2) - \sup_{h,\Tilde{h} \in \mathcal{H}} \Vert \Tilde{h} - h \Vert_{\temp_1} - 2 \epsilon \Big\vert \geq \delta 
\end{align*}
This implies that $\Vert F_{\temp_1}( (\Tilde{h}_\epsilon)_\sharp \temp_1) - F_{\temp_2} ( (h_\epsilon)_{\sharp} \temp_2) \Vert_{\temp_1} \geq W_2((\Tilde{h}_\epsilon)_\sharp \temp_1, (h_\epsilon)_{\sharp} \temp_2) \geq \delta$.  So we are done.
\end{proof}

Notice that \cref{separability} above acts as a blueprint to controlling the degree of separation in the LOT embedding via the bounds on the function class $\mathcal{H}$.  For the specific setting of the set of shears above, given a desired degree of separation $0 < \delta < W_2(\temp_1,\temp_2)$, we can choose $M$, $M_b$, and $\gamma$ in the definition of $\mathcal{H}_{\gamma, M}$ that guarantees that $F_{\refe}(\mathcal{H}_{\gamma, M, M_b} \star \temp_1)$ and $F_{\refe}(\mathcal{H}_{\gamma, M, M_b} \star \temp_2)$ are $\delta$-separated.  This leads us to \cref{shearSeparabilityGeneral}:

\begin{corollary}\label{shearSeparabilityGeneral}
Consider  probability distributions $\temp_1$ and $\temp_2$ with Wasserstein distance $W_2(\temp_1, \temp_2)$, and let $\delta > 0$.  Let us denote $R_1 = \max_{x \in \text{supp}(\temp_1)} \Vert x \Vert_2$ and $R_2 = \max_{x \in \text{supp}(\temp_2)} \Vert x \Vert_2$.  Moreover, for $\epsilon > 0$, define
\begin{align*}
    \mathcal{H}_{\gamma, M, M_b, \epsilon} = \{ \Tilde{h} : \Vert h - \Tilde{h} \Vert_{\temp_i} < \epsilon , i \in \{1,2\}, h \in \mathcal{H}_{\gamma, M, M_b} \}
\end{align*}
as the $\epsilon$-tube around $\mathcal{H}_{\gamma, M, M_b}$.  We consider the following 2 cases:

\textbf{Case 1:}  Assume that $W_2(\temp_1, \temp_2) > (R_1 + R_2) + \delta + 2\epsilon$.  If $M_b$ is chosen such that
\begin{align*}
    0 < M_b &< \frac{W_2(\temp_1, \temp_2) - \delta - 2\epsilon - (R_1 + R_2)}{2},
\end{align*}
then choosing $M$ such that
\begin{align*}
    2 < M \leq \frac{W_2(\temp_1, \temp_2) - \delta - 2\epsilon - 2 M_b + (R_1+R_2)}{R_1 + R_2}
\end{align*}
ensures that $F_{\refe}(\mathcal{H}_{\gamma, M, M_b, \epsilon} \star \temp_1)$ and $F_{\refe}(\mathcal{H}_{\gamma, M, M_b, \epsilon} \star \temp_2)$ are $\delta$-separated.

\textbf{Case 2:}  Assume that $\delta + 2\epsilon < W_2(\temp_1, \temp_2) < (R_1 + R_2) + \delta + 2\epsilon$.  If $M_b$ is chosen such that
\begin{align*}
    \max\Big\{ 0, \frac{W_2(\temp_1, \temp_2) - 2\epsilon - \delta - (R_1 +R_2)}{2} \Big\} < M_b  < \frac{W_2(\temp_1, \temp_2) - 2\epsilon - \delta}{2},
\end{align*}
then either choosing $M$ such that
\begin{align*}
    1 < M \leq \frac{W_2(\temp_1, \temp_2) - \delta - 2\epsilon - 2 M_b + (R_1+R_2)}{R_1 + R_2}
\end{align*}
or choosing $\gamma$ such that
\begin{align*}
    \gamma \geq \frac{2M_b + 2\epsilon + \delta - W_2(\temp_1, \temp_2) + R_1 + R_2}{R_1 + R_2}
\end{align*}
ensures that $F_{\refe}(\mathcal{H}_{\gamma, M, M_b, \epsilon} \star \temp_1)$ and $F_{\refe}(\mathcal{H}_{\gamma, M, M_b, \epsilon} \star \temp_2)$ are $\delta$-separated.
\end{corollary}

\begin{proof}[Proof of \Cref{shearSeparabilityGeneral}]\label{ProofshearSeparability}
From the lemma above, we need to only bound $\Vert h - I \Vert_{\refe}$ appropriately and invert the bounds.  First, note that because $h \in \mathcal{H}_{\gamma, M, M_b}$ can be written as the gradient of a convex function and $\mathcal{H}_{\gamma, M, M_b}$ is a convex set, we do satisfy the setting of the lemma.  Moreover, we know that the compatibility condition holds, which implies that
\begin{align*}
    W_2(h_\sharp \temp, \temp) &= \Vert h - I \Vert_\temp = \Bigg( \int \Vert (A-I)x + b \Vert_2^2 d\temp(x) \Bigg)^{\frac{1}{2}}  \\
    &\leq \underbrace{\Vert (A - I) x \Vert_\temp}_{I_1} + \underbrace{\Vert b \Vert_\temp}_{I_2}.
\end{align*}
Let us bound $I_1$ and $I_2$ separately.  For the bound of $I_1$, we have
\begin{align*}
    I_1 &= \bigg( \int \Vert (A - I) x \Vert_2^2 d\temp(x) \bigg)^{\frac{1}{2}} \\
    &\leq \bigg( \int \lambda_{\max}(A - I)^2  \Vert x \Vert_2^2 d\temp(x) \bigg)^{\frac{1}{2}} \\
    &\leq \bigg( \int \lambda_{\max}(A - I)^2  \max_{x \in \text{supp}(\temp)} \Vert x \Vert_2^2 d\temp(x) \bigg)^{\frac{1}{2}} \\
    &= \lambda_{\max}(A - I) \max_{x \in \text{supp}(\temp)} \Vert x \Vert_2 \underbrace{\bigg( \int d\temp(x) \bigg)^{\frac{1}{2}}}_{1} \\
    &\leq  \max\{\vert M - 1 \vert , \vert 1 - \gamma \vert \} \max_{x \in \text{supp}(\temp)} \Vert x \Vert_2.
\end{align*}
For the bound of $I_2$, we have
\begin{align*}
    I_2 &= \Vert b \Vert_\temp = \bigg( \int \Vert b \Vert_2^2 d\temp(x) \bigg)^{\frac{1}{2}} \leq \bigg( \int M_b^2 d\temp(x) \bigg)^{\frac{1}{2}} = M_b.
\end{align*}
Thus, if $h \in \mathcal{H}_{\gamma, M, M_b}$, we have
\begin{align*}
    W_2(h_\sharp \temp, \temp) \leq \max\{\vert M - 1 \vert , \vert 1 - \gamma \vert  \} \max_{x \in \text{supp}(\temp)} \Vert x \Vert_2 +  M_b.
\end{align*}
Using this for our specific choice of $\temp_1$ and $\temp_2$, we find that for $\Tilde{h}, \Tilde{h}^\star \in \mathcal{H}_{\gamma, M, M_b, \epsilon}$, we have
\begin{align*}
    W_2(\temp_1, \temp_2) - W_2(\Tilde{h}_\sharp \temp_1, \temp_1) - W_2(\Tilde{h}^\star_\sharp \temp_2, \temp_2)
\end{align*}
is lower bounded (via equation 2.1 of \cite{ambrosio2013user}) by
\begin{align*}
    W_2(\temp_1, \temp_2) - W_2(h_\sharp \temp_1, \temp_1) - \underbrace{W_2(h_\sharp \temp_1 , \Tilde{h}_\sharp \temp_1)}_{\leq \Vert h - \Tilde{h} \Vert_{\temp_1} < \epsilon } - W_2(h^\star_\sharp \temp_2, \temp_2) - \underbrace{W_2(h^\star_\sharp \temp_2, \Tilde{h}^\star_\sharp \temp_2)}_{\leq \Vert h^\star \Tilde{h}^\star \Vert_{\temp_2} < \epsilon } \\
    \geq W_2(\temp_1, \temp_2) - W_2(h_\sharp \temp_1, \temp_1) - W_2(h^\star_\sharp \temp_2, \temp_2) - 2 \epsilon ,
\end{align*}
which in turn is lower bounded by
\begin{align*}
     W_2(\temp_1, \temp_2) - 2M_b - \max\{\vert M - 1 \vert , \vert 1 - \gamma \vert \} \Big(\underbrace{\max_{x \in \text{supp}(\temp_1)} \Vert x \Vert_2 + \max_{x \in \text{supp}(\temp_2)} \Vert x \Vert_2}_{R_1 + R_2}\Big) - 2\epsilon.
\end{align*}
Now we just need to find sufficient conditions $M, M_b$, and $\gamma$ such that
\begin{align*}
    W_2(\temp_1, \temp_2) - 2M_b - \max\{\vert M - 1\vert , \vert 1 - \gamma\vert \} (R_1+ R_2 ) - 2\epsilon &\geq \delta > 0. & (\star)
\end{align*}
Notice that when $\vert M - 1 \vert > \vert 1 - \gamma \vert$, then $M > 1$ since $M$ is the bound on the largest eigenvalue.  Moreover, note that we cannot have $\gamma - 1 > M - 1$ since $\gamma < M$; thus, the only cases we need to consider are when $M - 1 > 1 - \gamma$ and $M - 1 < 1 - \gamma$.  We handle these cases separately.

\textbf{Case 1 ($M - 1 > 1 - \gamma$):}  Note that in this case we can rewrite $(\star)$ as
\begin{align*}
    M(R_1 + R_2) &\leq W_2(\temp_1, \temp_2) - 2M_b - 2\epsilon - \delta + (R_1+ R_2 ) \\
    M &\leq \frac{W_2(\temp_1, \temp_2) - 2M_b - 2\epsilon - \delta + (R_1+ R_2 )}{R_1 + R_2}.
\end{align*}

\textbf{Case 2 ($1 - \gamma > M - 1$):}  In this case, we can rewrite $(\star)$ as 
\begin{align*}
    \gamma(R_1 +R_2) &\geq \delta + 2M_b + 2\epsilon + (R_1+ R_2 ) - W_2(\temp_1, \temp_2) \\
    \gamma &\geq \frac{\delta + 2M_b + 2\epsilon + (R_1+ R_2 ) - W_2(\temp_1, \temp_2)}{R_1 + R_2}.
\end{align*}
Now we will investigate conditions in which case 1 and case 2 are active.

First note that if
\begin{align*}
    \frac{W_2(\temp_1, \temp_2) - 2M_b - 2\epsilon - \delta + (R_1+ R_2 )}{R_1 + R_2} &> 2 \\
    \iff \frac{W_2(\temp_1,\temp_2) - 2\epsilon - \delta - (R_1 + R_2) }{2} > M_b &> 0 \\
    \iff W_2(\temp_1, \temp_2) > \delta + 2\epsilon + (R_1 + R_2),
\end{align*}
we know that the first case is ensured since we can pick $M > 2$.  In this case, $M - 1 > \vert 1 - \gamma \vert$.  To see this, we see that if $0 < \gamma < 1$, then $M-1 > 2 - 1 = 1 > 1 - \gamma > 0$.  If $1 < \gamma < 2$, we again have that $M > \gamma$ implies that $M - 1 > \gamma - 1$.  Thus, in this regime, the choice of $M$ dominates.

Now if we want $1 < M < 2$, we find that
\begin{align*}
    M \leq \frac{W_2(\temp_1, \temp_2) - 2M_b - 2\epsilon - \delta + (R_1+ R_2 )}{R_1 + R_2} < 2 \\
    \iff \frac{W_2(\temp_1,\temp_2) - 2\epsilon - \delta - (R_1 + R_2) }{2} < M_b.
\end{align*}
Notice that since $M_b \geq 0$, if $W_2(\temp_1, \temp_2) < 2 \epsilon + \delta + (R_1 + R_2)$, then we definitely have the above inequality.  On the other hand,
\begin{align*}
    \frac{W_2(\temp_1, \temp_2) - 2M_b - 2\epsilon - \delta + (R_1+ R_2 )}{R_1 + R_2} > 1 \\
    \iff \frac{W_2(\temp_1,\temp_2) - 2\epsilon - \delta }{2} > M_b > 0 \\
    \iff W_2(\temp_1, \temp_2) > 2\epsilon + \delta.
\end{align*}
So we can pick appropriate $M_b$ such that
\begin{align*}
    \max\Big\{ \frac{W_2(\temp_1,\temp_2) - 2\epsilon - \delta - (R_1 + R_2) }{2}, 0 \Big\} < M_b < \frac{W_2(\temp_1,\temp_2) - 2\epsilon - \delta}{2},
\end{align*}
and in this case, we pick an appropriate $M$ such that
\begin{align*}
    1 < \frac{W_2(\temp_1, \temp_2) - 2\epsilon - 2M_b - \delta + (R_1+ R_2 )}{R_1 + R_2} \leq M < 2.
\end{align*}
In the case when $\vert 1 - \gamma \vert > \vert M - 1 \vert$ case, notice that
\begin{align*}
    \frac{\delta + 2\epsilon + 2M_b + (R_1+ R_2 ) - W_2(\temp_1, \temp_2)}{R_1 + R_2} < 1 \iff \exists M_b < \frac{W_2(\temp_1, \temp_2) - \delta - 2\epsilon }{2};
\end{align*}
thus, we can pick $\gamma$ such that
\begin{align*}
    1 > \gamma \geq \max\bigg\{ \frac{\delta + 2M_b + 2\epsilon + (R_1+ R_2 ) - W_2(\temp_1, \temp_2)}{R_1 + R_2}, 0 \bigg\}.
\end{align*}
So we still can satisfy the conditions for linear separability in these cases. \end{proof}

\section{Multiple References Example}\label{multipleEx}

\begin{example}
Recall the setup of \Cref{multipleLOT_Ex}, where we have two template distributions $\mu_1 = \mathcal{N}(0, \Sigma_1)$ and $\mu_2 = \mathcal{N}(0, \Sigma_2)$, a set of shears 
\begin{align*}
    \mathcal{H} = \{ A x : A =A^\top \in \RR^{n\times n}, M I_n \succeq A \succeq m I_n \succ 0 \}
\end{align*}
as our set of transformations, and reference distributions defined to be of the form $\refe_1 = (h_1)_\sharp \temp_1$ and $\refe_2 = (h_2)_\sharp \temp_2$ for $h_1(x) = A_1 x$ and $h_2(x) = A_2 x$ for $h_1, h_2 \in \mathcal{H}$ so that
\begin{align*}
    \refe_1 = (h_1)_\sharp \mu_1 = \mathcal{N}( 0, A_1 \Sigma_1 A_1^\top), \hspace{0.4cm} \refe_2 = (h_2)_\sharp \mu_2 = \mathcal{N}( 0, A_2 \Sigma_2 A_2^\top).
\end{align*}
Using exercise 6.3.1 of \cite{vershynin2018}, the bounds on our function class $\mathcal{H}$ is given by
\begin{align*}
    \sup_{A \in \mathcal{H}} \Vert (A - I) \Vert_{\temp_j} &= \sup_{A \in \mathcal{H}} \bigg(\mathbb{E}_{\temp_j} \Big[ \Vert (A - I)x \Vert_2^2 \Big]\bigg)^{1/2} = \sup_{A \in \mathcal{H}} \Vert (A - I) \Sigma_j^{1/2} \Vert_F \\
    &\leq \sup_{A \in \mathcal{H}} \Vert A - I \Vert_2 \Vert \Sigma_j^{1/2} \Vert_F \leq \max \big( \vert M - 1\vert , \vert 1 - m \vert \big) \max_{j} \Vert \Sigma_j^{1/2} \Vert_F \\
    &= L.
\end{align*}
To ensure separation, we use $L \leq \frac{W_2(\temp_1, \temp_2) - \delta}{2}$, which implies that
\begin{align*}
    \max \big( \vert M - 1\vert , \vert 1 - m \vert \big) \leq \frac{ \Tr( \Sigma_1 + \Sigma_2 - 2(\Sigma_1^{\frac{1}{2}} \Sigma_2 \Sigma_1^{\frac{1}{2}} )^{1/2} )^{1/2} - \delta }{2\max_{j = 1,2} \Vert \Sigma_j^{1/2} \Vert_F}.
\end{align*}
It is easy to see that $\frac{ W_2(\temp_1, \temp_2) }{2 \max_{j=1,2} \Vert \Sigma_j^{1/2} \Vert_F } < 1$.  This shows the bounds on $M$ and $m$ of \Cref{multipleLOT_Ex}.  Now notice that 
\begin{align*}
    T_{\sigma_j}^{h_\sharp \mu_i } =  (A_j \Sigma_j A_j)^{-1/2} ( (A_j \Sigma_j A_j)^{1/2} (A \Sigma_i A^\top) (A_j \Sigma_j A_j)^{1/2})^{1/2} (A_j \Sigma_j A_j)^{-1/2} x ;
\end{align*}
thus, for $h, \Tilde{h} \in \mathcal{H}$ where $h(x) = A x $ and $\Tilde{h}(x) = \Tilde{A} x $, we get $\Vert T_{\sigma_j}^{h_\sharp \mu_1} - T_{\sigma_j}^{ \Tilde{h}_\sharp \mu_2} \Vert_{\sigma_j}^2$ is equal to
\begin{align*}
    \mathbb{E}\Big[ \Vert S_j^{-1/2}(( S_j^{1/2} (A_1 \Sigma_1 A_1^\top) S_j^{1/2} )^{1/2} - ( S_j^{1/2} (A_2 \Sigma_2 A_2^\top) S_j^{1/2} )^{1/2}) S_j^{-1/2} x \Vert_2^2 \Big],
\end{align*}
where $S_j = A_j \Sigma_j A_j$ and the expectation is with respect to $\refe_j$.  Because $S_j^{-1/2} x = (A_j \Sigma_j A_j)^{-1/2} x \sim \mathcal{N}(0 , I)$ and exercise 6.3.1 of \cite{vershynin2018}, we find that the expectation above is equal to
\begin{align*}
     \Vert S_j^{-1/2}(( S_j^{1/2} (A_1 \Sigma_1 A_1^\top) S_j^{1/2} )^{1/2} - ( S_j^{1/2} (A_2 \Sigma_2 A_2^\top) S_j^{1/2} )^{1/2}) \Vert_F^2.
\end{align*}
Using the Courant-Fischer min-max theorem as explained in \cite{hornandjohnson} and our bounds on the eigenvalues of $h \in \mathcal{H}$, we can see that
\begin{align*}
    \frac{1}{m} \Sigma_j^{-1/2} \succeq (A_j \Sigma_j A_j)^{-1/2} \succeq \frac{1}{M} \Sigma_j^{-1/2}.
\end{align*}
Since $M^2 \Sigma_i \succeq (A_i \Sigma_i A_i) \succeq \epsilon^2 \Sigma_i$, we have
\begin{align*}
    ( (A_j \Sigma_j A_j)^{1/2} \underbrace{(A_1 \Sigma_1 A_1^\top)}_{\succeq m^2 \Sigma_1} (A_j \Sigma_j A_j)^{1/2} )^{1/2} - ( (A_j \Sigma_j A_j)^{1/2} \underbrace{(A_2 \Sigma_2 A_2^\top)}_{\preceq m^2 \Sigma_2} (A_j \Sigma_j A_j)^{1/2} )^{1/2} \\
    \succeq ( m^4 \Sigma_j^{1/2} \Sigma_1 \Sigma_j^{1/2} )^{1/2} - (m^4 \Sigma_j^{1/2} \Sigma_2 \Sigma_j^{1/2} )^{1/2} = m^2 \bigg[ (\Sigma_j^{1/2} \Sigma_1 \Sigma_j^{1/2} )^{1/2} - (\Sigma_j^{1/2} \Sigma_2 \Sigma_j^{1/2} )^{1/2}  \bigg],
\end{align*}
and similarly,
\begin{align*}
    ((A_j \Sigma_j A_j)^{1/2} \underbrace{(A_1 \Sigma_1 A_1^\top)}_{\preceq M^2 \Sigma_1} (A_j \Sigma_j A_j)^{1/2} )^{1/2} - ( (A_j \Sigma_j A_j)^{1/2} \underbrace{(A_2 \Sigma_2 A_2^\top)}_{\preceq M^2 \Sigma_2} (A_j \Sigma_j A_j)^{1/2} )^{1/2} \\
    \succeq ( M^4 \Sigma_j^{1/2} \Sigma_1 \Sigma_j^{1/2} )^{1/2} - (M^4 \Sigma_j^{1/2} \Sigma_2 \Sigma_j^{1/2} )^{1/2} = M^2 \bigg[ (\Sigma_j^{1/2} \Sigma_1 \Sigma_j^{1/2} )^{1/2} - (\Sigma_j^{1/2} \Sigma_2 \Sigma_j^{1/2} )^{1/2}  \bigg].
\end{align*}
This means that $\Vert T_{\sigma_j}^{h_\sharp \mu_1} - T_{\sigma_j}^{ \Tilde{h}_\sharp \mu_2} \Vert_{\sigma_j}^2$ has the following bounds
\begin{align*}
    \bigg\Vert \frac{M^2 }{m} \Sigma_j^{-1/2} \Big[ (\Sigma_j^{1/2} \Sigma_1 \Sigma_j^{1/2} )^{1/2} - (\Sigma_j^{1/2} \Sigma_2 \Sigma_j^{1/2} )^{1/2}  \Big] \bigg\Vert_F^2 \geq \Vert T_{\sigma_j}^{h_\sharp \mu_1} - T_{\sigma_j}^{ \Tilde{h}_\sharp \mu_2} \Vert_{\sigma_j}^2 \\
    \Vert T_{\sigma_j}^{h_\sharp \mu_1} - T_{\sigma_j}^{ \Tilde{h}_\sharp \mu_2} \Vert_{\sigma_j}^2 \geq \bigg\Vert \frac{m^2 }{M} \Sigma_j^{-1/2} \Big[ (\Sigma_j^{1/2} \Sigma_1 \Sigma_j^{1/2} )^{1/2} - (\Sigma_j^{1/2} \Sigma_2 \Sigma_j^{1/2} )^{1/2}  \Big] \bigg\Vert_F^2.
\end{align*}
Moreover, notice that since $\Sigma_j = \Sigma_1$ or $\Sigma_j = \Sigma_2$, we can assume without loss of generality that $\Sigma_j= \Sigma_1$ so that
\begin{align*}
    \Sigma_j^{-1/2} \bigg[ (\Sigma_j^{1/2} \Sigma_1 \Sigma_j^{1/2} )^{1/2} - (\Sigma_j^{1/2} \Sigma_2 \Sigma_j^{1/2} )^{1/2}  \bigg] =  \Sigma_1^{1/2} - \Sigma_1^{-1/2}(\Sigma_1^{1/2} \Sigma_2 \Sigma_1^{1/2} )^{1/2}.
\end{align*}
We can show, however, that the Frobenius norm of the right-hand side is actually $W_2(\temp_1, \temp_2)$.  To see this, first notice that because $\Sigma_1^{1/2} \Sigma_2 \Sigma_1^{1/2}$ is symmetric, using the cyclic property of traces, we have
\begin{align*}
    \Vert \Sigma_1^{-1/2} (\Sigma_1^{1/2} \Sigma_2 \Sigma_1^{1/2})^{1/2} \Vert_F^2 &= \Tr( (\Sigma_1^{1/2} \Sigma_2 \Sigma_1^{1/2})^{1/2} \Sigma_1^{-1/2} \Sigma_1^{-1/2} (\Sigma_1^{1/2} \Sigma_2 \Sigma_1^{1/2})^{1/2} ) \\ 
    &= \Tr( \Sigma_1^{-1} \Sigma_1^{1/2} \Sigma_2 \Sigma_1^{1/2}) = \Tr(\Sigma_2).
\end{align*}
Applying this result, we have that $\Vert \Sigma_1^{1/2} - \Sigma_1^{-1/2}(\Sigma_1^{1/2} \Sigma_2 \Sigma_1^{1/2} )^{1/2} \Vert_F^2$ is equal to
\begin{align*}
    \Vert \Sigma_1^{1/2} \Vert_F^2 + \underbrace{\Vert \Sigma_1^{-1/2} (\Sigma_1^{1/2} \Sigma_2 \Sigma_1^{1/2})^{1/2} \Vert_F^2}_{\Tr(\Sigma_2)} - 2 \Tr( \Sigma_1^{1/2} \Sigma_1^{-1/2}(\Sigma_1^{1/2} \Sigma_2 \Sigma_1^{1/2} )^{1/2}) \\
    = \Tr( \Sigma_1 + \Sigma_2 - 2 (\Sigma_1^{1/2} \Sigma_2 \Sigma_1^{1/2}) ) = W_2(\temp_1, \temp_2)^2 .
\end{align*}
So we get that
\begin{align*}
    \frac{M^2}{m} W_2(\temp_1,\temp_2) \geq \Vert T_{\sigma_j}^{h_\sharp \mu_1} - T_{\sigma_j}^{ \Tilde{h}_\sharp \mu_2} \Vert_{\sigma_j} \geq \frac{m^2}{M} W_2(\temp_1,\temp_2)
\end{align*}
for our choices of reference distributions, of which there are infinite choices because our choices of $\epsilon$ and $M$ are constrained by
\begin{align*}
    1 - \frac{ W_2(\temp_1, \temp_2) }{2 \max_{j=1,2} \Vert \Sigma_j^{1/2} \Vert_F } \leq m \leq M \leq 1 + \frac{ W_2(\temp_1, \temp_2) }{2 \max_{j=1,2} \Vert \Sigma_j^{1/2} \Vert_F } .
\end{align*}
\end{example}

\section{The shearing transformation}
\label{shearing}
\begin{algorithm}
\begin{algorithmic}[1]
\STATE \textbf{Inputs} : $28 \times 28$ matrix of pixel values corresponding to the image, matrix $A$ and shift $b$.
\STATE \textbf{Output} : A $28 \times 28$ matrix of pixel values corresponding to the transformation $A(x - center) + b$. (Center of the image is assumed to be $(14,14)$). Here $x = (i,j)$ where $i,j \in \{ 1,2, \cdots 28 \}$.\\
\STATE \texttt{ShearedImage} $\leftarrow$ An empty $28 \times 28$ array 
\FOR {$i=1, \cdots 28$}
    \FOR {$j=1, \cdots 28$}
       \STATE $y \leftarrow (i,j) - center$ 
       \STATE $x \leftarrow A^{-1}(y-b) + center$
       \IF {$x_1 > 28$ or $x_1<=0$ or $x_2 > 28$ or $x_2<=0$ }
            \STATE \texttt{ShearedImage(i,j)} $\leftarrow$ 0
        \ELSE
            \STATE \texttt{ShearedImage(i,j)} $\leftarrow$ Interpolation of the pixel values (of the original image) of the four grid points corresponding to the grid box  which $x$ belongs to.
        \ENDIF
    \ENDFOR
\ENDFOR
\RETURN \texttt{ShearedImage}
\end{algorithmic}
\caption{Procedure to produce shears of an image}
\label{shears}
\end{algorithm}
Following notation introduced in Section 4 of the main text, the function class $\mathcal{H}$ with respect to which we perform numerical experiments on MNIST images to study linear separability is,
\begin{equation}\label{shearsH}
    \mathcal{H} = \Big\{ Ax + b : \text{$A$ is symmetric positive definite}, \: b \in \mathbb{R}^2 \Big\},
\end{equation}
Specifically we choose $A$ to be,
\begin{equation}
    A  = \begin{pmatrix}
    \cos \theta & -\sin \theta\\
    \sin \theta & \cos \theta
    \end{pmatrix}^T
    \begin{pmatrix}
    \lambda_1 & 0\\
    0 & \lambda_2
    \end{pmatrix}
    \begin{pmatrix}
    \cos \theta & -\sin \theta\\
    \sin \theta & \cos \theta
    \end{pmatrix}
\end{equation}
where, $\lambda_1, \lambda_2 >0$ so that $A$ is positive definite. In the subsequent sections, we present the classification results for two different choices for the range of parameter ($\lambda_1, \lambda_2, \theta, b$) values, one representing a mild shearing of the images and the other representing a severe shearing of the images.
%%%%%%%%%%%%%%%%%%%%%%%%%%%%%%
\section{Standard deviation in test error of MNIST classification experiments}
\begin{figure}[H]
    \centering
    \includegraphics[width=\textwidth]{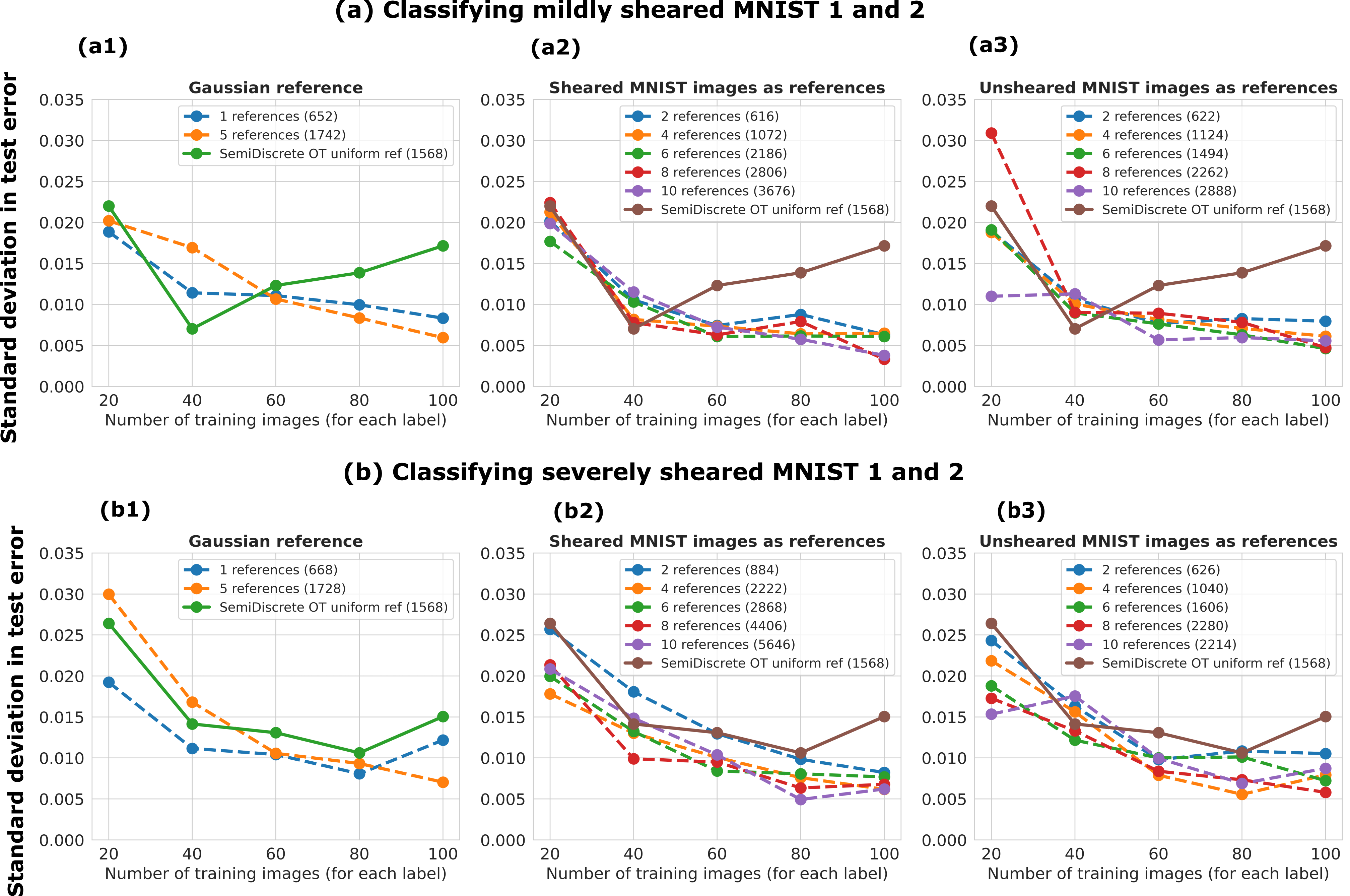}
    \caption{\scriptsize{(a) Standard deviation in test errors for binary classification of mildly sheared MNIST 1s and 2s using (a1) Gaussian references (a2) sheared MNIST 1s and 2s as references (a3) unsheared MNIST 1s and 2s as references. (b) Standard deviation in test errors for binary classification of severely sheared MNIST 1s and 2s using (b1) Gaussian references (b2) sheared MNIST 1s and 2s as references (b3) unsheared MNIST 1s and 2s as references. In the cases where MNIST images are used as references, the results are reported for the cases where the number of references used is $2i$ for $i=1, \cdots 5$ wherein $i$ images from each class are randomly drawn to be used as references from a pool of images that do not correspond to any of the training and testing images. For each fixed number of training images per class, $N_{train}$, the mean test classification error averaged across 20 random choices of $N_{train}$ training images (per class) and $1000$ test images (per class) is reported. The number inside the parenthesis in the legends of the images denote the length of the LOT feature vector corresponding to the particular choice of references. In all figures, for comparison, the results for classification using the semi discrete linear optimal transport framework \cite{merigot20} which uses the uniform measure as the reference is also reported.}}
    \label{multirefImage12std}
\end{figure}

\begin{figure}[H]
    \centering
    \includegraphics[width=\textwidth]{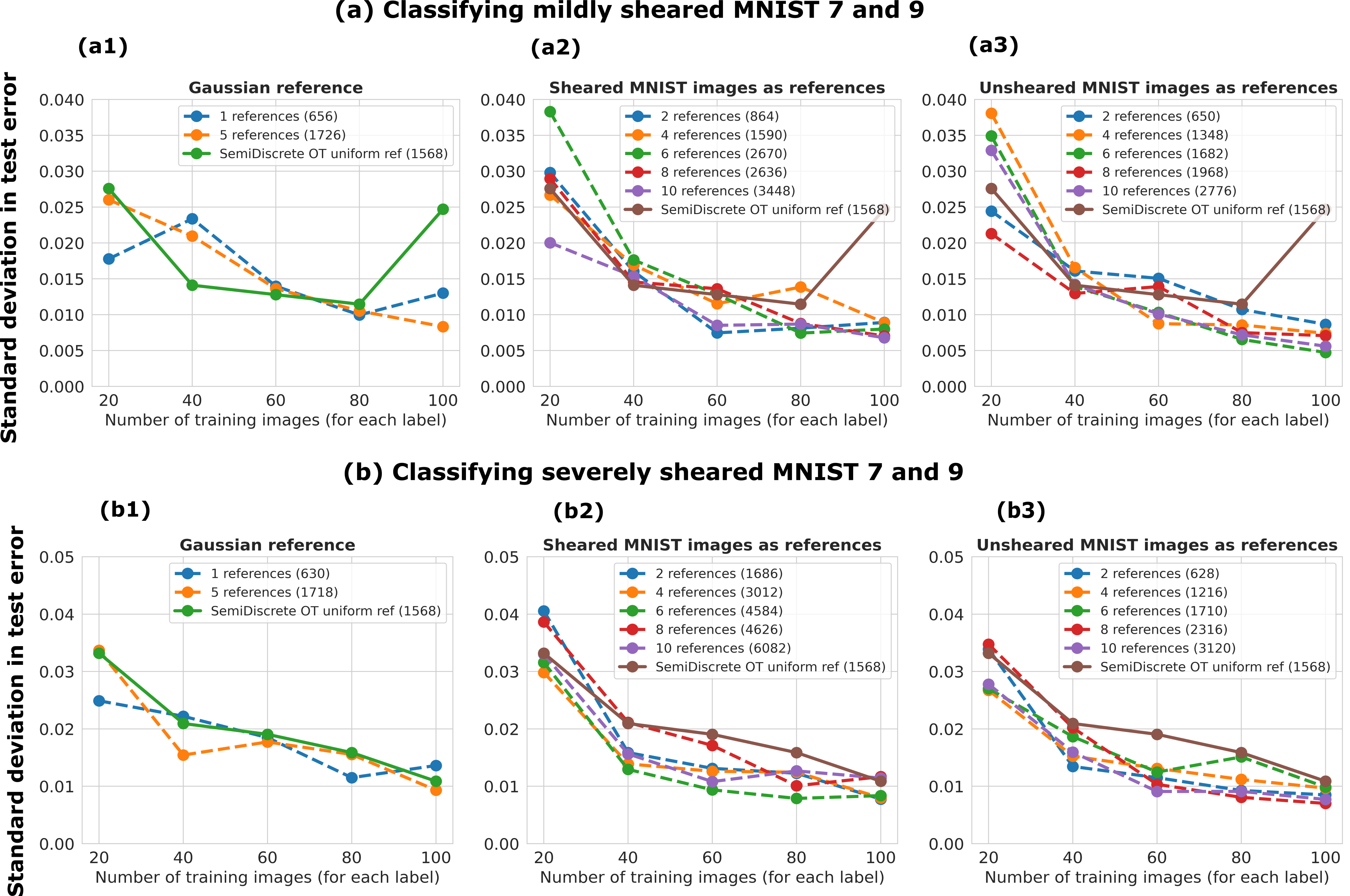}
    \caption{\scriptsize{(a) Standard deviation in test errors for binary classification of mildly sheared MNIST 7s and 9s using (a1) Gaussian references (a2) sheared MNIST 7s and 9s as references (a3) unsheared MNIST 7s and 9s as references. (b) Standard deviation in test errors for binary classification of severely sheared MNIST 7s and 9s using (b1) Gaussian references (b2) sheared MNIST 7s and 9s as references (b3) unsheared MNIST 7s and 9s as references. In the cases where MNIST images are used as references, the results are reported for the cases where the number of references used is $2i$ for $i=1, \cdots 5$ wherein $i$ images from each class are randomly drawn to be used as references from a pool of images that do not correspond to any of the training and testing images. For each fixed number of training images per class, $N_{train}$, the mean test classification error averaged across 20 random choices of $N_{train}$ training images (per class) and $1000$ test images (per class) is reported. The number inside the parenthesis in the legends of the images denote the length of the LOT feature vector corresponding to the particular choice of references. In all figures, for comparison, the results for classification using the semi discrete linear optimal transport framework \cite{merigot20} which uses the uniform measure as the reference is also reported.}}
    \label{multirefImage79std}
\end{figure}

\section{Numerical validation of example \ref{example11}}
\label{example11appendix}
To illustrate \Cref{shearTheorem}, we had provided a simple example with Gaussians (see example \ref{example11} of main text).  Let $\temp = \mathcal{N}(m_1, I_n)$.  Consider a symmetric positive definite matrix $A$ with spectral decomposition $A = P^\top \Lambda P$ and a corresponding fixed shear $S(x) = A x + b$ for some $b \in \RR^n$, which yields the pushforward $S_\sharp \mu = \mathcal{N}(A m_1 + b, AA^\top)$.  For simplicity, we will check that the subset of compatible affine transformations 
\begin{equation}
\begin{split}
    \mathcal{F}_{\text{affine}}(P) &= \{ f(x) = C x + d  : f \in \mathcal{F}(P)\} \\
    &= \{ P^\top D P x + d : D_{ij} = 0 \hspace{0.1cm}\forall\hspace{0.1cm} i \neq j,  D_{ii} > 0, d \in \RR^n \}
\end{split}
\label{Faffinepp}
\end{equation}

yields reference distributions $\refe \in \{f_\sharp \mu : f \in \mathcal{F}_{\text{affine}}(P) \}$ so that the compatibility condition hold.  In particular note that for $f(x) = Cx + d = P^\top D P x + d$, our reference distributions have the form
\begin{align*}
    \refe = \mathcal{N}( Cm_1 + d, C C^\top) = \mathcal{N}( C m_1 + d, P^\top D^2 P).
\end{align*}

\begin{figure}[H]
    \centering
    \includegraphics[width=0.75\textwidth]{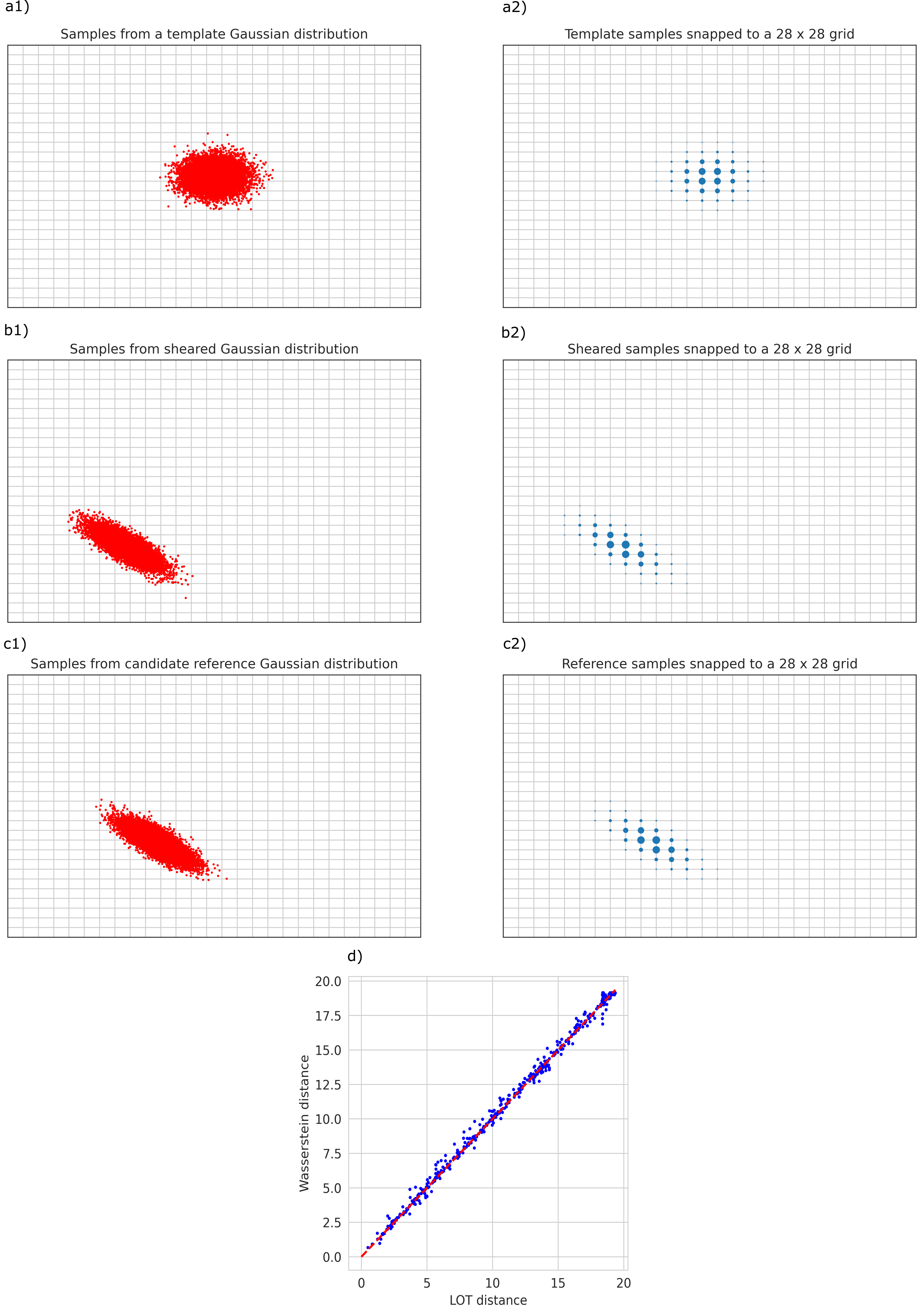}
    \caption{\scriptsize{a1) Samples from a Gaussian distribution that serves as the template $\mu$. a2) Approximation of the template distribution as a discrete distribution on a grid. b1) Samples from sheared distribution $S_{\sharp}\mu$. b2) Approximation of the sheared distribution as a discrete distribution on a grid. c1) Samples from a candidate referencec distribution $f_{\sharp}\mu \in \mathcal{F}_{\text{affine}}(P)$ (equation \ref{Faffinepp}). c2) Approximation of the reference distribution as a discrete distribution on a grid. d) Numerical validation of the equivalence of LOT distance $W_{2,f_{\sharp \mu}}^{LOT}(\mu,S_{\sharp} \mu) $ and the Wasserstein distance $W_2(\mu,S_{\sharp} \mu)$ under compatibility as in example \ref{example11} using the LOT framework for different choices of shear $S$.}}
    \label{example11Illustration}
\end{figure}

\section{Comparison with Convolutional Neural Networks (CNNs)}
\begin{figure}[H]
    \centering
    \includegraphics[width=\textwidth]{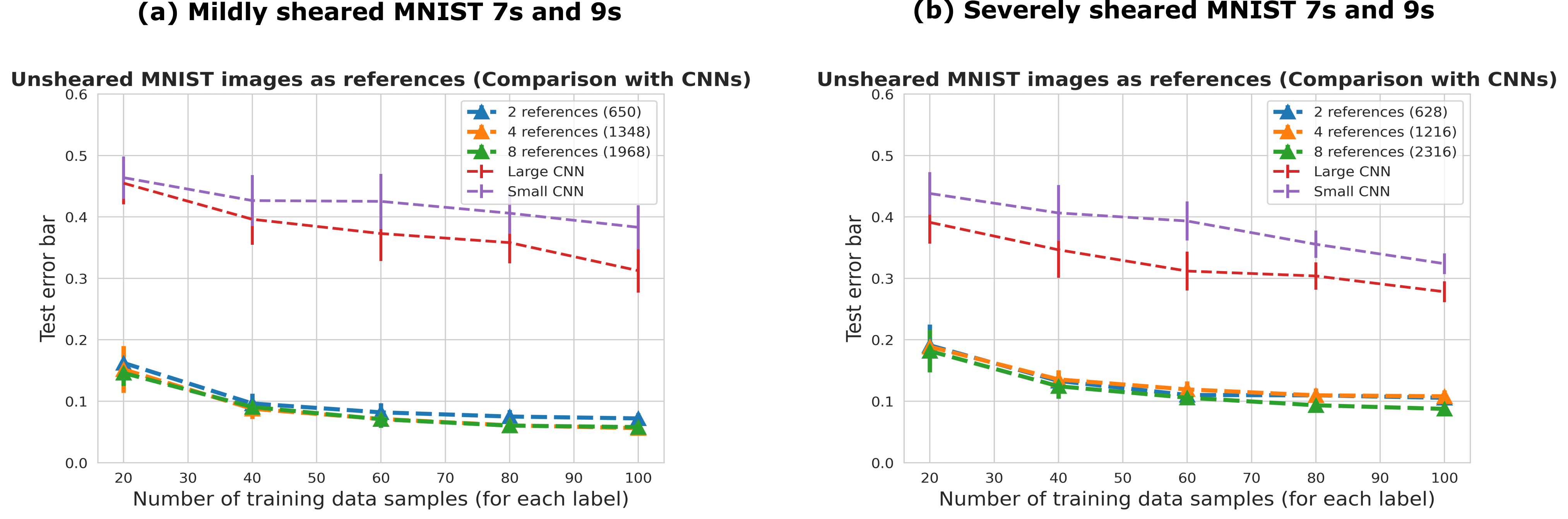}
    \caption{\scriptsize{Comparison of discrete LOT classification of (a) mildly sheared MNIST 7s and 9s (b) severely sheared MNIST 7s and 9s with convolutional neural network with 1586 training parameters (labelled small CNN) and 3650 training parameters (labelled large CNN) under identical training and testing conditions.}}
    \label{CNNcomparison}
\end{figure}

\end{appendix}
\end{document}